\documentclass[11pt,twoside,reqno]{amsart}

\usepackage{microtype}
\usepackage[OT1]{fontenc}
\usepackage{type1cm}
\usepackage{amssymb}
\usepackage{mathtools}
\usepackage{esint}
\usepackage{nicefrac}
\usepackage[left=2.3cm,top=3cm,right=2.3cm]{geometry}

\usepackage[active]{srcltx}
\usepackage{verbatim}

\usepackage[pdftex]{graphicx}
\usepackage{caption}
\usepackage{subfigure}
\usepackage{color}
\usepackage{tikz}
\usetikzlibrary{trees,patterns,calc}

\numberwithin{equation}{section}

\theoremstyle{plain}
\newtheorem{theorem}{Theorem}[section]
\newtheorem{maintheorem}{Theorem}

\newtheorem{proposition}[theorem]{Proposition}
\newtheorem{lemma}[theorem]{Lemma}

\theoremstyle{remark}
\newtheorem{remark}[theorem]{Remark}

\newtheorem{example}[theorem]{Example}

\theoremstyle{definition}

\newcommand{\HH}{\mathcal{H}}
\newcommand{\PP}{\mathcal{P}}

\newcommand{\R}{\mathbb{R}}

\newcommand{\N}{\mathbb{N}}

\newcommand{\iii}{\mathtt{i}}
\newcommand{\jjj}{\mathtt{j}}

\newcommand{\eps}{\varepsilon}
\newcommand{\fii}{\varphi}
\newcommand{\roo}{\varrho}
\newcommand{\ualpha}{\overline{\alpha}}
\newcommand{\lalpha}{\underline{\alpha}}

\newcommand{\CL}{C_{\textnormal{left}}}
\newcommand{\CR}{C_{\textnormal{right}}}
\newcommand{\TL}{T_{\textnormal{left}}}
\newcommand{\TR}{T_{\textnormal{right}}}

\renewcommand{\epsilon}{\varepsilon}

\newcommand{\la}{\underline{\alpha}}
\newcommand{\ua}{\overline{\alpha}}

\newcommand{\II}{\mathcal{I}}
\newcommand{\JJ}{\mathcal{J}}

\newcommand{\ii}{\mathtt{i}}

\newcommand{\jj}{\mathtt{j}}

\newcommand{\iin}[1]{\ii|_{#1}}
\newcommand{\jjn}[1]{\jj|_{#1}}

\newenvironment{labeledlist}[2][\unskip]
{ \renewcommand{\theenumi}{({#2}\arabic{enumi}{#1})}
  \renewcommand{\labelenumi}{({#2}\arabic{enumi}{#1})}
  \begin{enumerate} }
{ \end{enumerate} }

\DeclareMathOperator{\udimm}{\overline{dim}_M}
\DeclareMathOperator{\ldimm}{\underline{dim}_M}

\DeclareMathOperator{\dimh}{dim_H}

\DeclareMathOperator{\dima}{dim_A}

\DeclareMathOperator{\cdima}{\mathcal{C}dim_A}
\DeclareMathOperator{\cdimh}{\mathcal{C}dim_H}

\DeclareMathOperator{\dist}{dist}
\DeclareMathOperator{\diam}{diam}
\DeclareMathOperator{\diag}{diag}
\DeclareMathOperator{\proj}{proj}

\DeclareMathOperator{\spt}{spt}
\DeclareMathOperator{\length}{length}

\begin{document}

\title{Rigidity of quasisymmetric mappings on self-affine carpets}

%\title[Rigidity of quasisymmetric mappings on self-affine carpets]{Rigidity of quasisymmetric mappings and conformal dimension on self-affine carpets}

\author{Antti K\"aenm\"aki}
\address[Antti K\"aenm\"aki]{
        Department of Mathematics and Statistics \\
        P.O.\ Box 35 (MaD) \\
        FI-40014 University of Jyv\"askyl\"a \\
        Finland}
\email{antti.kaenmaki@jyu.fi}

\author{Tuomo Ojala}
\address[Tuomo Ojala]{
        Department of Mathematics and Statistics \\
        P.O.\ Box 35 (MaD) \\
        FI-40014 University of Jyv\"askyl\"a \\
        Finland}
\email{tuomo.ojala@gmail.com}

\author{Eino Rossi}
\address[Eino Rossi]{
        Department of Mathematics and Statistics \\
        P.O.\ Box 35 (MaD) \\
        FI-40014 University of Jyv\"askyl\"a \\
        Finland}
\email{eino.rossi@gmail.com}

%\thanks{}
\date{\today}

\begin{abstract}
  We show that the class of quasisymmetric maps between horizontal self-affine carpets is rigid. Such maps can only exist when the dimensions of the carpets coincide, and in this case, the quasisymmetric maps are quasi-Lipschitz. We also show that horizontal self-affine carpets are minimal for the conformal Assouad dimension.
\end{abstract}

\maketitle

\section{Introduction} \label{sec:intro}

We consider the following two general questions:
\begin{itemize}
  \item[(1)] By understanding the fine structure of sets, is it possible to say anything about quasisymmetic mappings between the sets?
  \item[(2)] What kind of sets are minimal for the conformal dimension?
\end{itemize}
In the class of horizontal self-affine carpets we can answer both of the questions. Our method to prove the results builds on the analysis of weak tangent sets and mappings. This is done in general metric spaces and therefore, should hold an independent interest and also provide a framework for possible further applications.

If $(X,d)$ and $(Y,\roo)$ are metric spaces and $\eta \colon [0,\infty) \to [0,\infty)$ is a homeomorphism, then a homeomorphism $f \colon X \to Y$ is \emph{$\eta$-quasisymmetric} if
\begin{equation*}
  \frac{\roo(f(x),f(y))}{\roo(f(x),f(z))} \le \eta\biggl( \frac{d(x,y)}{d(x,z)} \biggr)
\end{equation*}
for all $x,y,z \in X$ with $x \ne z$. Quasisymmetric mappings are a non-trivial generalization of bi-Lipschitz mappings. While bi-Lipschitz maps shrink or expand the diameter of a set by no more than a multiplicative factor, quasisymmetric maps satisfy the weaker geometric property that they preserve the relative sizes of sets: if two sets $A$ and $B$ have diameters $t$ and are no more than distance $t$ apart, then the ratio of their sizes changes by no more than a multiplicative constant.

The \emph{Assouad dimension} of a set $E \subset X$, denoted by $\dima(E)$, is the infimum of all $t$ satisfying the following: there exists a constant $C \ge 1$ such that each set $E \cap B(x,R)$ can be covered by at most $C(r/R)^{-t}$ balls of radius $r$ centered at $E$ for all $0<r<R$. The \emph{conformal Assouad dimension} of $E$ is
\begin{equation*}
  \cdima(E) = \inf\{ \dima(E') : E' \text{ is a quasisymmetric image of } E \}.
\end{equation*}
It is worth emphasizing that the codomains of the quasisymmetric mappings used in the definition can be any metric spaces. A set $E$ is \emph{minimal} for the conformal Assouad dimension if $\dima(E) = \cdima(E)$. We remark that conformal dimension and minimality can similarly be defined also for other set dimensions.

Concerning the question (1), Bonk and Merenkov \cite[Theorem 1.1]{BonkMerenkov2013} have recently shown that every quasisymmetric self-map on the standard $\tfrac13$-Sierpi\'nski carpet is an isometry. The \emph{$\tfrac13$-Sierpi\'nski carpet} is a planar self-similar set satisfying the open set condition obtained by a repetitive process where, at each step, the cube is divided in nine subcubes and the middle one is removed. Concerning the question (2), Mackay \cite[Theorem 1.4]{Mackay2011} has shown that a Gatzouras-Lalley carpet $E$ is minimal when its projection onto the horizontal coordinate is a line segment; otherwise $\cdima(E)=0$. \emph{Gatzouras-Lalley carpets} are self-affine sets constructed by a repetitive process in which, using the same pattern at each step, the rectangle is first partitioned into vertical tubes and then from each tube a collection of disjoint subrectangles is chosen in such a way that the vertical length of each subrectangle is strictly smaller than the horizontal length which equals the width of the tube.

Theorems \ref{thm:horizontal_carpet} and \ref{thm:conformal_dimensio} below give an answer for both of the questions in the class of horizontal self-affine carpets. A set in this class is constructed by a repetitive process in which, using the same pattern at each step, in a given rectangle, we choose a collection of disjoint subrectangles having width longer than height such that every vertical line going through the rectangle intersects at least two such subrectangles. The precise definition and the proofs of the results are given in \S \ref{sec:horizontal-carpets}.

\begin{maintheorem} \label{thm:horizontal_carpet}
  If $E$ and $F$ are horizontal self-affine carpets, then any quasisymmetric mapping $f \colon E \to F$ is quasi-Lipschitz.
\end{maintheorem}

A mapping $f \colon X \to Y$ is \emph{quasi-Lipschitz} if
\begin{equation*} %\label{eq:mainlemma}
  \frac{\log\roo(f(x),f(y))}{\log d(x,y)} \to 1
\end{equation*}
uniformly as $d(x,y) \to 0$. The class of quasi-Lipschitz mappings is strictly more general than the class of bi-Lipschitz mappings. For example, $f\colon [0,1]\to\R$ defined by $f(0)=0$ and $f(x)=x\log x$ for all $x\in(0,1]$ is quasi-Lipschitz, but not Lipschitz. In Lemma \ref{lem:quasiliphausdorff}, we show that quasi-Lipschitz mappings preserve the Hausdorff dimension and the upper and lower Minkowski dimensions. Therefore, two horizontal self-affine carpets either have the same dimension or there does not exist a quasisymmetric mapping between them. This is in a huge contrast to the self-similar case. Wang, Wen, and Zhu \cite{WangWenZhu2010} have shown that all self-similar sets satisfying the strong separation condition are quasisymmetrically equivalent. 

\begin{maintheorem} \label{thm:conformal_dimensio}
  Horizontal self-affine carpets are minimal for the conformal Assouad dimension.
\end{maintheorem}

In Example \ref{ex:gap}, we show that a self-affine carpet can be minimal even if the projection onto the horizontal coordinate is not a line. Furthermore, in Remark \ref{rem:projection-interval}, we point out that, to obtain Theorem \ref{thm:conformal_dimensio}, instead of the vertical line condition it suffices to assume that the projection of the carpet onto the horizontal coordinate is a line segment. Therefore, under the strong separation condition, Theorem \ref{thm:conformal_dimensio} strictly generalizes the result of Mackay \cite[Theorem 1.4]{Mackay2011}.

The proof of Theorem \ref{thm:horizontal_carpet} has three essential steps. At first, in Theorem \ref{thm:quasilip}, we show that a quasisymmetric mapping is quasi-Lipschitz provided that its quasisymmetric weak tangent maps are bi-Lipschitz. The second step is to find a geometric condition for the weak tangent sets under which quasisymmetric weak tangent maps are bi-Lipschitz. This is done in Lemma \ref{lem:manyfiberedcomponents} by modifying the result of Le Donne and Xie \cite[Theorem 1.1]{LeDonneXie2015} for our purposes. We show that quasisymmetric maps between finite unions of fibered spaces are bi-Lipschitz. It is worth emphasizing that both of these results hold in general metric spaces. Finally, in Theorem \ref{thm:BK_general}, we show that the weak tangent sets of horizontal self-affine carpets are finite unions of fibered spaces. This generalizes the result of Bandt and K\"aenm\"aki \cite[Theorem 1]{BandtKaenmaki2013}. It is also an essential ingredient in the proof of Theorem \ref{thm:conformal_dimensio}. To prove Theorem \ref{thm:conformal_dimensio}, it therefore remains to show that any compact set, whose weak tangents are minimal, is minimal for the conformal Assouad dimension.

\section{Convergence of metric spaces and mappings}\label{sec:MappingPackages}

The purpose of this section is to introduce weak tangents for metric spaces and quasisymmetric mappings. To that end, let us start fixing some notation. We say that a metric space $(X,d)$ is \emph{doubling} with a constant $N \in \N$ if any closed ball $B_X(x,r)$ with center $x \in X$ and $r>0$ can be covered by $N$ balls of radius $r/2$. If the underlying metric space is clear from the context, we write $B(x,r)$ instead of $B_X(x,r)$. Recall that if a metric space $(X,d)$ is doubling, then, according to the Assouad embedding theorem, for any $0<\alpha<1$, $(X,d^\alpha)$ can be mapped onto a subset of some Euclidean space by a bi-Lipschitz mapping.
We say that a metric space $(X,d)$ is \emph{uniformly perfect} with a constant $D \ge 1$ if for all $x\in X$ and $r>0$ we have $B(x,r) \setminus B(x,r/D) \ne \emptyset$ whenever $X \setminus B(x,r) \ne \emptyset$.

We continue by recalling some of the definitions from the book of David and Semmes \cite[\S 8]{DavidSemmes1997}. Let $(F_i)$ be a sequence of non-empty closed subsets of $\R^d$. We say that $F_i$ \emph{converges} to a nonempty closed $F \subset \R^d$ if
\begin{equation*}
  \lim_{i \to \infty} \sup_{x \in F_i \cap B(0,R)} \dist(x,F) = 0
\end{equation*}
and
\begin{equation*}
  \lim_{i \to \infty} \sup_{y \in F \cap B(0,R)} \dist(y,F_i) = 0
\end{equation*}
for all $R>0$. In this case, if $(X,d)$ is a metric space, and $\fii_i \colon F_i \to X$ and $\fii \colon F \to X$, then we say that $\fii_i$ \emph{converges} to $\fii$ if for each sequence $(x_i)$ in $\R^d$ for which $x_i \in F_i$ for all $i$ and $\lim_{i \to \infty} x_i = x \in F$ we have
\begin{equation*}
  \lim_{i \to \infty} \fii_i(x_i) = \fii(x).
\end{equation*}
Note that the limit function $\fii$ is unique.

Recall that a \emph{pointed metric space} is a triple $(X,d,p)$, where $(X,d)$ is a complete doubling metric space and $p \in X$ is a fixed \emph{basepoint}. Suppose that $(X,d,p)$ and $(X_i,d_i,p_i)$ are pointed metric spaces for all $i$ such that all the metric spaces are doubling with the same constant. By the Assouad embedding theorem, choose $0<\alpha\le 1$ and bi-Lipschitz embeddings $h_i \colon (X_i,d_i^\alpha) \to (\R^d,|\cdot|)$ and $h \colon (X,d^\alpha) \to (\R^d,|\cdot|)$ so that $h_i(p_i)=0$ for all $i$ and $h(p)=0$. We say that $(X_i,d_i,p_i)$ \emph{converges} to $(X,d,p)$ if, for some such choices $h_i$ and $h$, $h_i(X_i)$ converges to $h(X)$ in $\R^d$ and $(x,y) \mapsto d_i(h_i^{-1}(x),h_i^{-1}(y))$ defined on $h_i(X_i) \times h_i(X_i)$ converges to $(x,y) \mapsto d(h^{-1}(x),h^{-1}(y))$ on $h(X) \times h(X)$. By \cite[Lemma 8.12]{DavidSemmes1997}, the limit $(X,d,p)$ is unique up to an isometry that respects the basepoint. Observe also that, by \cite[Lemma 8.13]{DavidSemmes1997}, a sequence of pointed metric spaces has a converging subsequence.

A \emph{mapping package} is a triplet $((X,d,p),(Y,\roo,q),f)$ consisting of two pointed metric spaces and a mapping $f \colon X \to Y$ for which $f(p)=q$. Suppose that $\PP = ((X,d,p),(Y,\roo,q),f)$ and $\PP_i = ((X_i,d_i,p_i),(Y_i,\roo_i,q_i),f_i)$ are mapping packages for all $i$. We say that $\PP_i$ \emph{converges} to $\PP$ if $(X_i,d_i,p_i)$ converges to $(X,d,p)$, $(Y_i,\roo_i,q_i)$ converges to $(Y,\roo,q)$, and $g_i \circ f_i \circ h_i^{-1}$ converges to $g \circ f \circ h^{-1}$, where $h,g,h_i,g_i$ are the Assouad embeddings of $X,Y,X_i,Y_i$, respectively, as above. By \cite[Lemma 8.21]{DavidSemmes1997}, the limit $\PP$ is unique up to isometries that respect the basepoints. Observe also that, by \cite[Lemma 8.22]{DavidSemmes1997}, if $\PP_i = ((X_i,d_i,p_i),(Y_i,\roo_i,q_i),f_i)$ are mapping packages for all $i$ such that the mappings $f_i$ are equicontinuous and uniformly bounded on bounded sets and all the metric spaces are doubling with the same constant, then $\PP_i$ has a converging subsequence. Recall that $\{ f_i \}$ is \emph{equicontinuous on bounded sets} if for every $R>0$ and $\eps>0$ there is $\delta>0$ such that
\begin{equation*}
  \roo_i(f_i(x),f_i(y)) < \eps
\end{equation*}
for all $i$ and for all $x,y \in B_{X_i}(p_i,R)$ with $d_i(x,y)<\delta$. Furthermore, $\{ f_i \}$ is \emph{uniformly bounded on bounded sets} if
\begin{equation*}
  \sup_i\sup_{x_i \in B_{X_i}(p_i,R)} \roo_i(f_i(x_i),q_i) < \infty
\end{equation*}
for all $R > 0$.

\begin{lemma} \label{lem:quasisymmetric}
  Suppose that $\PP_i = ((X_i,d_i,p_i), (Y_i,\roo_i,q_i) , f_i)$, where $f_i\colon X_i\to Y_i$ is $\eta$-quasisymmetric, is a mapping package for all $i \in \N$ such that all the metric spaces are doubling with the same constant. Assume further that there exists a constant $C \ge 1$ and a sequence $(w_i)$ such that $w_i \in X_i$,
  \begin{equation*}
    1/C \leq  d_i(p_i, w_i) \leq C, \quad \text{and} \quad 1/C \leq \roo_i( q_i, f_i(w_i)) \leq C
  \end{equation*}
  for all $i \in \N$. Then, after passing to a subsequence, the mapping packages $\PP_i$ converges to a mapping package $((X,d,p), (Y,\roo,q) , f)$, where $f\colon X\to Y$ is $\eta$-quasisymmetric. Moreover, there exists $w \in X$ such that
  \begin{equation*}
    1/C \leq  d(p, w) \leq C, \quad \text{and} \quad 1/C \leq \roo( q, f(w)) \leq C.
  \end{equation*}
\end{lemma}

\begin{proof}
  As remarked above, the convergence of mapping packages follow once we verify that the family $\{f_i\}$ is equicontinuous and uniformly bounded on bounded sets. The argument here can be compared to \cite[Theorems 2.21 and 3.4]{TukiaVaisala1980} and \cite[Corollary 10.30]{Heinonen2001}. Let us start with the equicontinuity. Fix $R,\eps>0$. Let $ \delta > 0 $ be so that  $ \delta < (4C)^{-1} $ and $ \eta\left( 4C \delta \right) \eta\left( (C+R) C \right) C <\eps$. Now $d_i( x_i , p_i) > (4C)^{-1}$ or $d_i( x_i , w_i) > (4C)^{-1}$ for all $x_i\in B_{X_i}(p_i,R)$. Let us first assume that $d_i( x_i , w_i) > (4C)^{-1}$. Since $\eta$ is increasing, we have for all $x_i,y_i\in B_{X_i}(p_i,R)$ with $d_i(x_i,y_i)<\delta$ that
  \begin{align*}
    \roo_i( f_i(x_i) , f_i(y_i) )
    &=\frac{ \roo_i( f_i(x_i) , f_i(y_i) ) }{ \roo_i( f_i(x_i) , f_i(w_i) ) }
      \frac{ \roo_i( f_i(x_i) , f_i(w_i) ) }{ \roo_i( f_i(p_i) , f_i(w_i) ) } \roo_i( f_i(p_i) , f_i(w_i) ) \\
    &\leq \eta\left( \frac{d_i(x_i,y_i) }{ d_i(x_i,w_i) } \right) \eta\left( \frac{d_i(x_i,w_i) }{ d_i(p_i,w_i) } \right) \roo_i( q_i , f_i(w_i) )    \\
    &\leq \eta\left( 4C \delta \right) \eta\left( (C+R) C \right) C \leq \eps.
  \end{align*}
  In the case $d_i( x_i , p_i) > (4C)^{-1}$ we get the same estimate just by switching the roles of $p_i$ and $w_i$ above. We have thus verified the equicontinuity on bounded sets. The uniform boundedness on bounded sets follows immediately since for fixed $R>0$ and $x_i\in B_{X_i}(p_i,R)$ we have
  \begin{align*}
    \roo_i( f_i(x_i) , q_i )
    &=\frac{ \roo_i( f_i(x_i) , f_i(p_i) ) }{ \roo_i( f_i(w_i) , f_i(p_i) ) } \roo_i( f_i(w_i) , q_i ) \\
    &\leq \eta\left( \frac{d_i(x_i,p_i) }{ d_i(w_i,p_i) } \right) C  
    \leq \eta\left( RC \right) C.
  \end{align*}
  Therefore, there exists a subsequence along which the mapping packages converge. In what follows, we keep denoting the subsequence by the original sequence.

  It remains to show that $f \colon X \to Y$ is $\eta$-quasisymmetric. Since the mapping packages $\PP_i$ converge there are Assouad embeddings $h$, $g$, $h_i$, $g_i$ of $X$, $Y$, $X_i$, $Y_i$, respectively, such that
  \begin{equation} \label{embeddingsconverge}
    g_i \circ f_i \circ h_i^{-1} \to g \circ f \circ h^{-1}.
  \end{equation}
  Furthermore, since both pointed metric spaces of $\PP_i$ converge we also have
  \begin{equation} \label{distancesconverge}
  \begin{split}
    d_i(h_i^{-1} (\,\cdot\,) , h_i^{-1} (\,\cdot\,)) &\to d(h^{-1} (\,\cdot\,) , h^{-1} (\,\cdot\,)), \\
    \roo_i(g_i^{-1} (\,\cdot\,) , g_i^{-1} (\,\cdot\,)) &\to \roo(g^{-1} (\,\cdot\,) , g^{-1} (\,\cdot\,)).
  \end{split}
  \end{equation}
  Fix three different points $x,y,z\in X$. Pick three sequences $h_i(x_i)$, $h_i(y_i)$, $h_i(z_i)$ converging to $h(x)$, $h(y)$, $h(z)$, respectively. By \eqref{embeddingsconverge}, we see that $g_i\circ f_i(x_i)$, $g_i\circ f_i(y_i)$, $g_i\circ f_i(z_i)$ converge to $g\circ f(x)$, $g\circ f(y)$, $g\circ f(z)$, respectively. Since each $f_i$ is $\eta$-quasisymmetric we have
  \begin{equation*}
    \frac{\roo_i( f_i(x_i) , f_i(y_i) )}{\roo_i( f_i(x_i) , f_i(z_i) )} \leq \eta\left( \frac{d_i(x_i,y_i)}{d_i(x_i,z_i)} \right).
  \end{equation*}
  Therefore, since $f_i= g_i^{-1} \circ g_i \circ f_i$ and $\mathrm{Id} = h_i^{-1} \circ h_i$, we have, by \eqref{distancesconverge} and the continuity of $\eta$, that
  \begin{equation} \label{eq:limitquasisymmetry}
  \frac{\roo( f(x) , f(y) )}{\roo( f(x) , f(z) )} \leq \eta\left( \frac{d(x,y)}{d(x,z)} \right).
  \end{equation}
  This implies that $f$ is a continuous injection. If it were a bijection, then it is elementary to see that $f^{-1}$ satisfies \eqref{eq:limitquasisymmetry} with $1/\eta^{-1}(t^{-1})$ in place of $\eta(t)$. Hence $f^{-1}$ is continuous and $f$ is $\eta$-quasisymmetric. So the only thing left to prove is that $f$ is surjective.

  To that end, fix $y\in Y$. Pick a sequence $g_i(y_i)$ converging to $g(y)$. By \eqref{distancesconverge}, there is $i_0$ such that $\roo_i(y_i,q_i) \le 2\roo(y,q)$ for all $i \ge i_0$. Similarly as in the beginning of the proof, we see that $\{ f_i^{-1} \}$ is uniformly bounded on bounded sets. Thus there exists $C \in \R$ such that $d_i(f_i^{-1}(y_i),p_i) \le C$ for all $y_i \in B_{Y_i}(q_i,2\roo(y,q))$ for all $i \ge i_0$. Letting $x_i = f_i^{-1}(y_i)$ we have $x_i \in B_{X_i}(p_i,C)$ for all $i \ge i_0$. Since the sequence $(h_i(x_i))$ is contained in a compact subset of some Euclidean space it has a converging subsequence. Denoting the limit point of this subsequence by $z$, we have, by \eqref{embeddingsconverge}, that $f(h^{-1}(z)) = y$. This is what we wanted to show.
  
  Finally, since the Assouad embeddings $h_i$ are bi-Lipschitz with the same constant, the points $h_i(w_i)$ are contained in an annulus centered at the origin. Therefore there exists a converging subsequence. By definitions, this finishes the proof.
\end{proof}

We will next apply the notions of convergence to define weak tangents for metric spaces and quasisymmetric mappings. Let $(X,d)$ be a complete doubling metric space. We say that $(\hat X,\hat d,p)$ is a \emph{weak tangent} of $(X,d)$ if there are a sequence $(p_i)$ of points in $X$ and a sequence $(r_i)$ of positive reals converging to zero such that $(X,d/r_i,p_i)$ converges to $(\hat X,\hat d,p)$. Note that from each given $(p_i)$ and $(r_i)$ one can extract a subsequence along which there exists a weak tangent. If $(p_i)$ is a constant sequence, then weak tangents are called \emph{tangents}.

Furthermore, let $(Y,\roo)$ be another complete doubling metric space and $f \colon X \to Y$. If $(\hat X, \hat d,p)$ and $(\hat Y,\hat\roo,q)$ are weak tangents of $(X,d)$ and $(Y,\roo)$, respectively, then we say that $\hat f \colon \hat X \to \hat Y$ is a \emph{weak tangent mapping} of $f$ if the mapping packages $((X,d/r_i,p_i),(Y,\roo/t_i,q_i),f)$ converge to $((\hat X,\hat d,p),(\hat Y,\hat\roo,q),\hat f)$. Observe that if $f$ is $\eta$-quasisymmetric, then by choosing a sequence $(w_i)$ of points in $X$ such that $d(p_i,w_i)$ converges to zero and setting $r_i = d(p_i,w_i)$ and $t_i = \roo(f(p_i),f(w_i))$ for all $i$, Lemma \ref{lem:quasisymmetric} guarantees that there exists a subsequence along which the weak tangent mapping $\hat f$ of $f$ is $\eta$-quasisymmetric and a point $w \ne p$ in $\hat X$ such that
\begin{equation} \label{eq:1-eta-qs}
   \hat\roo(\hat f(p),\hat f(w)) = \hat d(p,w).
\end{equation}

\section{Rigidity of quasisymmetric mappings}\label{sec:rigidity}

In this section, we prove a rigidity result for quasisymmetric mappings between two metric spaces. We show that if the weak tangent mappings are bi-Lipschitz, then the original quasisymmetric mapping is quasi-Lipschitz. This result allows us to transfer the tangent level rigidity information back to the original quasisymmetric mapping. This is a useful observation since it is often the case that tangents are more regular than the original object. In \S \ref{sec:horizontal-carpets}, we will exhibit this phenomenon in a concrete setting.

\begin{theorem} \label{thm:quasilip}
  Suppose that $(X,d)$ and $(Y,\roo)$ are uniformly perfect compact doubling metric spaces. If $f \colon X \to Y$ is $\eta$-quasisymmetric such that any $\eta$-quasisymmetric weak tangent mapping $\hat f$ of $f$ satisfying \eqref{eq:1-eta-qs} for some $p$ and $w$ is $L$-bi-Lipschitz with $L$ depending only on $\eta$, then $f$ is quasi-Lipschitz.
\end{theorem}

Let us sketch the main idea of the proof. Assuming that $f$ is not quasi-Lipschitz, we find a sequence of pairs of points in which $f$ obeys a true H\"older behavior. In Lemma \ref{lem:helperLemma}, we show that for each pair we find a triplet of points, with comparable distances, in which $f$ still obeys a true H\"older behavior. These triplets then allow us to define weak tangets so that the limiting maps are quasisymmetric and the convergence of distances ensures that the true H\"older behavior is visible at the limit. This contradicts the assumption that the weak tangent mappings are bi-Lipschitz.

We will first recall a general lemma about the H\"older behavior of quasisymmetric mappings.

\begin{lemma} \label{lem:holderconstants}
  Suppose that $(X,d)$ and $(Y,\roo)$ are uniformly perfect bounded metric spaces. If $f$ is $\eta$-quasisymmetric, then there exist exponents $\Lambda\ge\lambda>0$ and constants $C\ge c>0$ depending only on $\eta$, the diameters of the spaces, and uniform perfectness constants such that
  \begin{equation*}
    cd(x,y)^\Lambda \le \roo(f(x),f(y)) \le Cd(x,y)^\lambda
  \end{equation*}
  for all $x,y \in X$.
\end{lemma}

\begin{proof}
  The lemma is stated in \cite[Corollary 11.5]{Heinonen2001}. By \cite[Theorem 11.3]{Heinonen2001}, there are constants $Q \ge 1$ and $0 < \beta \le 1$ depending only on $\eta$ and uniform perfectness constants such that $f$ is $\tilde \eta$-quasisymmetric where
  \begin{equation*}
    \tilde\eta(t) = Q\max\{ t^\beta, t^{1/\beta} \}
  \end{equation*}
  for all $t \ge 0$. Fix $x,y \in X$. Choose $z \in X$ such that $d(x,z)\geq \max\{d(x,y) , 2^{-1}\diam(X)\}$. Now
  \begin{align*}
    \roo(f(x),f(y))
    &\le Q \biggl( \frac{d(x,y)}{d(x,z)} \biggr)^{\beta} \roo(f(x),f(z)) \le Q 2^\beta \diam(X)^{-\beta} \diam(Y) d(x,y)^{\beta}.
  \end{align*}
  Note that the same estimates work even if we have to choose $z=y$. The other direction follows by using similar estimates for the inverse of $f$.
\end{proof}

The above result is needed in the proof of the following lemma which is a key observation to prove Theorem \ref{thm:quasilip}.

\begin{lemma} \label{lem:helperLemma}
  Suppose that $(X,d)$ and $(Y,\roo)$ are uniformly perfect (with constant $D \ge 1$) bounded metric spaces and $f \colon X \to Y$ is $\eta$-quasisymmetric. If $\kappa>1$ and $(x_i)$ and $(y_i)$ are sequences of points in $X$ such that
  \begin{equation} \label{eq:roo-cond}
    \roo(f(x_i),f(y_i)) < d(x_i,y_i)^{1+\kappa}
  \end{equation}
  for all $i \in \N$ and $d(x_i,y_i) \to 0$ as $i \to \infty$, then for every $0 < \eps_0 < 1$ there is $0<\eps<\eps_0$ such that for every $m \in \N$ there exists a triplet $(a,b,c)$ of three distinct points in $X$ such that
  \begin{enumerate}
    \item $d(a,b) \le 1/m$, \label{enum:ehto1}
    \item $\eps D^{-1}d(a,c) < d(a,b) \le \eps d(a,c)$, \label{enum:ehto2}
    \item $\roo(f(a),f(b)) < \eps^{1+\kappa/2} \roo(f(a),f(c))$. \label{enum:ehto3}
  \end{enumerate}
\end{lemma}

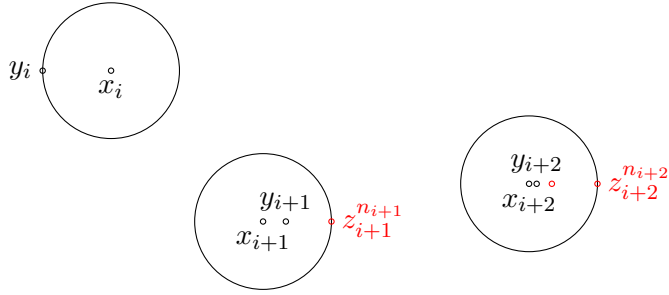
\begin{figure}[t]
  \begin{tikzpicture}
    %drawing from left to right
    \draw (-1,2) node[below] {$x_i$} circle (1pt);
    \draw (-1,2) circle (0.9cm);
    \draw (-1.9,2) node[left] {$y_i$} circle (1pt);
    \draw (1,0) node[below] {$x_{i+1}$} circle (1pt);
    \draw (1,0) circle (.9cm);
    \draw (1.3,0) node[above] {$y_{i+1}$} circle (1pt);
    \draw[red] (1.9,0) node[right] {$z^{n_{i+1}}_{i+1}$} circle (1pt);
    \draw (4.5,.5) node[below] {$x_{i+2}$} circle (1pt);
    \draw (4.5,.5) circle (.9cm);
    \draw (4.6,.5) node[above] {$y_{i+2}$}circle (1pt);
    \draw[red] (4.8,.5) circle (1pt);
    \draw[red] (5.4,.5) node[right] {$z^{n_{i+2}}_{i+2}$} circle (1pt);
  \end{tikzpicture}
  \caption{In the proof of Lemma \ref{lem:helperLemma}, we choose points $z^{j}_i$ for all $j\in\{0,\ldots,n_i\}$ so that the distance $d(x_i,z^{n_i}_i)$ is independent of $i$ and the relative distance $d(x_i,z^{j}_i)/d(x_i,z^{j+1}_i)$ is roughly $\eps$. Thus $n_i\to\infty$ as $d(x_i,y_i)\to 0$.}
  \label{fig:valinta}
\end{figure}

\begin{proof}
  Observe that, by Lemma \ref{lem:holderconstants}, there exist exponents $\Lambda \ge \lambda > 0$ and constants $C \ge c > 0$ such that
  \begin{equation} \label{eq:holderconstants}
    cd(x,y)^\Lambda \le \roo(f(x),f(y)) \le Cd(x,y)^\lambda
  \end{equation}
  for all $x,y \in X$. By \eqref{eq:roo-cond}, we have $\Lambda > 1+\kappa$.

  Fix $\eps_0>0$. Let us first show that for every $0<\eps<\eps_0$ and $m \in \N$ there exist triplets $(a,b,c)$ satisfying \eqref{enum:ehto1} and \eqref{enum:ehto2}. To that end, fix $0<\eps<\eps_0$ and $m \in \N$. Let $m_0 \ge m$ be such that $X \setminus B(x,\eps^{-1}D/m_0) \ne \emptyset$ for all $x \in X$ and choose $i \in \N$ so that
  \begin{equation} \label{eq:kanihatusta}
    d(x_i,y_i)^{\kappa/2} < \min\{ cm_0^{-(\Lambda-1-\kappa/2)}, m_0^{-\kappa/2} \}.
  \end{equation}
  Set $z_i^0 = y_i$ and, relying on the uniform perfectness, choose points $z_i^1,\ldots,z_i^{n_i}$ inductively so that
  \begin{equation*}
    \eps D^{-1}d(x_i,z_i^j) < d(x_i,z_i^{j-1}) \le \eps d(x_i,z_i^j)
  \end{equation*}
  for all $j \in \{ 1,\ldots,n_i \}$, where $n_i \in \N$ is such that
  \begin{equation*}
    d(x_i,z_i^{n_i-1}) < m_0^{-1} \le d(x_i,z_i^{n_i}).
  \end{equation*}
  See Figure \ref{fig:valinta} for an illustration. Note that, by \eqref{eq:kanihatusta}, such a number $n_i$ exists. Therefore, $X \setminus B(x_i, \eps^{-1}d(x_i,z_i^{j-1})) \supset X \setminus B(x_i, \eps^{-1}D/m_0) \ne \emptyset$ for all $j \in \{ 1,\ldots,n_i \}$ and the uniform perfectness condition is applicable. It now follows immediately that for each $j \in \{ 1,\ldots,n_i \}$ the triplet $(x_i,z_i^{j-1},z_i^j)$ satisfies the conditions \eqref{enum:ehto1} and \eqref{enum:ehto2}.
  
  Let us next assume to the contrary that for every $0<\eps<\eps_0$ there exists $m \in \N$ such that for every triplet $(a,b,c)$ satisfying \eqref{enum:ehto1} and \eqref{enum:ehto2}, the condition \eqref{enum:ehto3} fails, i.e.\
  \begin{equation} \label{eq:3-failed}
    \roo(f(a),f(b)) \ge \eps^{1+\kappa/2} \roo(f(a),f(c)).
  \end{equation} 
  Fix $0<\eps<\eps_0$ and let $m \in \N$ be as above. Then, in particular, the triplets $(x_i,z_i^{j-1},z_i^j)$ found in the first part of the proof satisfy \eqref{eq:3-failed} for all $j \in \{ 1,\ldots,n_i \}$. Now, using \eqref{eq:roo-cond}, \eqref{eq:3-failed}, \eqref{eq:holderconstants}, and \eqref{eq:kanihatusta}, we get
  \begin{align*}
    d(x_i,y_i)^{1+\kappa} &> \roo(f(x_i),f(y_i)) \ge (\eps^{1+\kappa/2})^{n_i} \roo(f(x_i),f(z_i^{n_i})) \\
    &= (\eps^{n_i}d(x_i,z_i^{n_i}))^{1+\kappa/2} \frac{\roo(f(x_i),f(z_i^{n_i}))}{d(x_i,z_i^{n_i})^{1+\kappa/2}} \\
    &\ge d(x_i,y_i)^{1+\kappa/2} cd(x_i,z_i^{n_i})^{\Lambda-1-\kappa/2} \\
    &\ge d(x_i,y_i)^{1+\kappa/2} cm_0^{-(\Lambda-1-\kappa/2)} \\
    &\ge d(x_i,y_i)^{1+\kappa}
  \end{align*}
  which is a contradiction; see Figure \ref{fig:contra} for an illustration. Thus, for some $j \in \{ 1,\ldots,n_i \}$, the triplet $(x_i,z_i^{j-1},z_i^j)$ satisfies the conditions \eqref{enum:ehto1}, \eqref{enum:ehto2}, and \eqref{enum:ehto3}.
\end{proof}

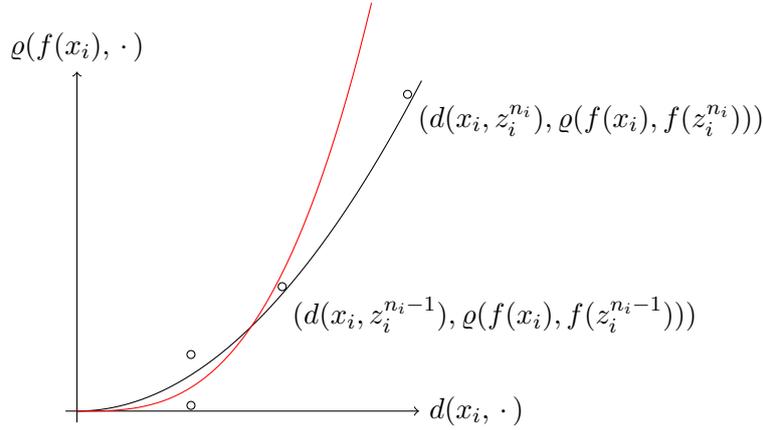
\begin{figure}[t]
  \begin{tikzpicture}[scale=1.5]
    \draw[->] (-.1,0) -- (3,0) node[right] {$d(x_i,\,\cdot\,)$};
    \draw[->] (0,-.1) -- (0,3) node[above] {$\roo(f(x_i),\,\cdot\,)$};
    \draw[xscale=5.5, yscale=29, domain=0:.55,smooth,variable=\x] plot ({\x},{\x*\x/3});
    \draw (2.9,2.8) node[below right] {$(d(x_i,z_i^{n_i}),\roo(f(x_i),f(z_i^{n_i})))$} circle (1pt);
    \draw (1.8,1.1) node[below right] {$(d(x_i,z_i^{n_i-1}),\roo(f(x_i),f(z_i^{n_i-1})))$} circle (1pt);
    \draw (1,.5) circle (1pt);
    \draw (1,.05) circle (1pt);
    \draw[xscale=5.5, yscale=29,domain=0:.47,smooth,variable=\x,red] plot ({\x},{1.2*\x^3});
  \end{tikzpicture}
  \caption{When stepping $j$ from $n_i$ to $0$, the ratios $d(x_i,z_i^j)/d(x_i,z_i^{j+1})$ are roughly $\eps$, meanwhile the ratios $\roo( f(x_i),f(z_i^j) )/\roo( f(x_i),f(z_i^{j+1}) )$ are larger than $\eps^{1+\kappa/2}$, since the triplets $(x_i,z_i^{j-1},z_i^j)$ satisfy \eqref{eq:3-failed}. Thus we are basically saying that the points $( d(x_i,z_i^{j}), \roo(f(x_i),f(z_i^{j})))$ stay above the graph of $x \mapsto Cx^{1+\kappa/2}$ for some constant $C$. Since we start from a distance $d(x_i,z^{n_i}_i)$ that is independent of $i$ and the space is compact, also $\roo(f(x_i),f(z^{n_i}_i))$ is roughly a constant and thus the constant $C$ is independent of $i$. On the other hand, our original assumption was that $\roo(f(x_i), f(y_i))<d(x_i,y_i)^{1+\kappa}$, so when we reach $z^0_i=y_i$, the point $( d(x_i,y_i), \roo(f(x_i),f(y_i)))$ should be below the graph of $x \mapsto Cx^{1+\kappa}$. This is where the contradiction arises, since for small $x$ the graph of $x\mapsto x^{1+\kappa/2}$ is below the graph of $x\mapsto Cx^{1+\kappa}$ and $d(x_i,y_i)/d(x_i,z_i^{n_i})$ gets arbitrarily small as $i\to\infty$.}
  \label{fig:contra}
\end{figure}

\begin{remark} \label{rem:weak-tangent}
  To prove Theorem \ref{thm:quasilip}, it seems to be essential to use weak tangents instead of tangents. Even though the sequence $(x_i)$ converges to some point $x$ (along some subsequence), it might happen that when looking at triplets $(x,x_i,y_i)$, the relative distance $d(x_i,y_i)/d(x,y_i)$ tends to zero and we would not see anything relevant about $f$ at the tangent at $x$. For this reason we need the triplets found in Lemma \ref{lem:helperLemma}.
\end{remark}

We are now ready to prove Theorem \ref{thm:quasilip}.

\begin{proof}[Proof of Theorem \ref{thm:quasilip}]
  Suppose to the contrary that there is an $\eta$-quasisymmetric mapping $f \colon X \to Y$ which is not quasi-Lipschitz. Then there is $\kappa>0$ such that for every $\delta>0$ there exist $x,y \in X$ with $d(x,y)<\delta$ such that
  \begin{equation*}
    \roo(f(x),f(y)) > d(x,y)^{1-\kappa} \quad \text{or} \quad \roo(f(x),f(y)) < d(x,y)^{1+\kappa}.
  \end{equation*}
  We may assume that $\roo(f(x),f(y)) < d(x,y)^{1+\kappa}$ since otherwise we can consider the inverse of $f$. Observe that the inverse of a weak tangent mapping of $f$ is a weak tangent mapping of $f^{-1}$. Thus there exist sequences $(x_i)$ and $(y_i)$ of points in $X$ such that
  \begin{equation*}
    \roo(f(x_i),f(y_i)) < d(x_i,y_i)^{1+\kappa}
  \end{equation*}
  for all $i \in \N$ and $d(x_i,y_i) \to 0$ as $i \to \infty$. Let $L \ge 1$ be the common bi-Lipschitz constant of $\eta$-quasisymmetric weak tangent mappings given by the assumption and let $D \ge 1$ be the uniform perfectness constant of $X$. By Lemma \ref{lem:helperLemma}, there exist $0<\eps<(2DL)^{-2/\kappa}$ and a sequence of triplets $(a_m,b_m,c_m)$ such that
  \begin{enumerate}
    \item $d(a_m,b_m) \le 1/m$,
    \item $\eps D^{-1}d(a_m,c_m) < d(a_m,b_m) \le \eps d(a_m,c_m)$,
    \item $\roo(f(a_m),f(b_m)) < \eps^{1+\kappa/2} \roo(f(a_m),f(c_m))$,
  \end{enumerate}
  for all $m \in \N$.
  
  By Lemma \ref{lem:quasisymmetric}, the mapping packages $((X,d/d(a_m,c_m),a_m), (Y,\roo/\roo(f(a_m),f(c_m)),f(a_m)), f)$ converge to a mapping package $((\hat X,\hat d,a), (\hat Y,\hat\roo,f(a)), \hat f)$ along some subsequence (which we keep denoting as the original sequence). By the convergence of the spaces, there are Assouad embeddings $h$, $g$, $h_m$, $g_m$ of $(\hat X,\hat d)$, $(\hat Y,\hat \roo)$, $(X,d/d(a_m,c_m))$, $(Y,\roo/\roo(f(a_m),f(c_m)))$, respectively; recall the proof of Lemma \ref{lem:quasisymmetric}. Observe that $h_m(a_m) \to h(a)$ and $g_m \circ f(a_m) \to g \circ \hat f(a)$ as $m \to \infty$ and likewise for $b_m$ and $c_m$. Moreover, $\hat f \colon \hat X \to \hat Y$ is $\eta$-quasisymmetric and $\hat \roo(\hat f(a),\hat f(c)) = \hat d(a,c)$. Hence, by the assumption, $\hat f$ is a bi-Lipschitz mapping with the constant $L$. Since the embeddings $g$ and $g_m$ are bi-Lipschitz with the same constants there exists $m_0 \in \N$ such that
  \begin{equation*}
    \hat\roo(\hat f(a),\hat f(b)) \le \frac{\roo(f(a_m),f(b_m))}{\roo(f(a_m),f(c_m))} + \eps^{1+\kappa/2}
  \end{equation*}
  for all $m \ge m_0$. Furthermore, by the convergence of the distances and \eqref{enum:ehto2}, we have $\eps D^{-1}\hat d(a,c) \le \hat d(a,b)$. Since $\hat d(a,c) = 1$ and $\hat f$ is $L$-bi-Lipschitz we get, by \eqref{enum:ehto3},
  \begin{align*}
    \eps D^{-1}L^{-1} &= \eps D^{-1}L^{-1}\hat d(a,c) \le L^{-1}\hat d(a,b) \le \hat\roo(\hat f(a),\hat f(b)) \\
    &\le \frac{\roo(f(a_m),f(b_m))}{\roo(f(a_m),f(c_m))} + \eps^{1+\kappa/2} \le 2\eps^{1+\kappa/2} < \eps D^{-1}L^{-1}.
  \end{align*}
  This contradiction finishes the proof.
\end{proof}

To finish this section, we show that quasi-Lipschitz mappings preserve dimension. We denote the Hausdorff, upper Minkowski, and lower Minkowski dimensions by $\dimh$, $\udimm$, and $\ldimm$, respectively.

\begin{lemma} \label{lem:quasiliphausdorff}
  If $(X,d)$ and $(Y,\roo)$ are separable metric spaces and $f \colon X \to Y$ quasi-Lipschitz, then $\dimh f(X) = \dimh X$. Furthermore, if $(X,d)$ and $(Y,\roo)$ are compact metric spaces, then $\udimm f(X) = \udimm X$ and $\ldimm f(X) = \ldimm X$.
\end{lemma}

\begin{proof}
  Note that since a quasi-Lipschitz mapping is locally invertible, $X$ is separable, and the Hausdorff dimension is countably stable, we may, without loss of generality assume that $f$ is invertible. Thus it suffices to show that $\dimh f(X) \le \dimh X$. Let $t>\dimh X$ and choose $\eps>0$ such that $(1-\eps)t > \dimh X$. It suffices to show that $\HH^t(f(X))<\infty$. By the quasi-Lipschitz assumption, there is $\delta_0>0$ such that
  \begin{equation*}
    d(x,y)^{1+\eps} \le \roo(f(x),f(y)) \le d(x,y)^{1-\eps}
  \end{equation*}
  whenever $d(x,y)<\delta_0$. Since $\HH^{(1-\eps)t}(X)=0$, we may choose for each $0<\delta<\delta_0$ a countable $\delta$-cover $\{ U_i \}$ of $X$ such that
  \begin{equation*}
    \sum_i \diam(U_i)^{(1-\eps)t} < 1.
  \end{equation*}
  Obviously the collection $\{ f(U_i) \}$ covers $f(X)$, so
  \begin{equation*}
    \sum_i \diam(f(U_i))^t \le \sum_i \diam(U_i)^{(1-\eps)t} < 1.
  \end{equation*}
  Since this holds for all $0<\delta<\delta_0$ we have shown that $\HH^t(f(X)) < \infty$.

  To show the second claim, observe that since $X$ is compact and the Minkowski dimension is finitely stable, we may again assume that $f$ is invertible. Let $t>\udimm X$ and choose $\eps>0$ such that $(1-\eps)t > \udimm X$. By the quasi-Lipschitz assumption, there is $\delta_0>0$ such that
  \begin{equation*}
    d(x,y)^{1+\eps} \le \roo(f(x),f(y)) \le d(x,y)^{1-\eps}
  \end{equation*}
  whenever $d(x,y)<\delta_0$. Now for all $r<\delta_0$, it holds that $f(B_X(x,r))\subset B_Y( f(x), r^{(1-\eps)} ) $. Therefore we get
  \begin{align*}
   (1-\eps) t
   &> \udimm X
    = \limsup_{r\downarrow 0} \frac{\log N( X , r^{ \frac{1}{1-\eps} } ) }{ -\log r^{ \frac{1}{1-\eps} } }
    \geq \limsup_{r\downarrow 0} \frac{\log N( f(X) , r ) }{ -\frac{1}{1-\eps}\log r }
    = (1-\eps)\udimm f(X)
  \end{align*}
  and similarly for $\ldimm$. Since $f$ was assumed to be invertible the other inequalities follow by the same estimate for $f^{-1}$.
\end{proof}

\section{Fibered spaces}
In this section, we present a condition for weak tangent sets that guarantees that the assumptions of Thorem \ref{thm:quasilip} are satisfied. The condition is a modification of a result of Le Donne and Xie \cite[Theorem 1.1]{LeDonneXie2015} for our purposes. We show that if the weak tangent sets are unions of fibered spaces, then the weak tangent mappings that map fibers onto fibers are bi-Lipschitz with the same constant. To this end, we introduce the following definition. If $(X,d)$ and $(Y,\roo)$ are metric spaces, $f \colon X \to Y$ is $\eta$-quasisymmetric, and there are points $p,w\in X$ and a constant $C \ge 1$ such that
\begin{equation} \label{eq:C-condition}
  C^{-1} \leq \frac{ \roo(f(p),f(w)) }{ d(p,w) } \le C,
\end{equation}
then we say that $f$ is $(C,\eta)$-\emph{quasisymmetric}. Note that a given quasisymmetric mapping clearly satisfies \eqref{eq:C-condition} for some $C$. The idea here is to show that the restriction of a $(1,\eta)$-quasisymmetric weak tangent mapping to each fibered space is a $(C,\eta)$-quasisymmetric mapping with the same constant $C$. This will then guarantee the existence of a uniform bi-Lipschitz constant in the assumptions of Theorem \ref{thm:quasilip}.

Let us first recall the definition of the fibered space from \cite{LeDonneXie2015}. Let $(X,d)$ be a metric space and $F,G \subset X$ closed sets. We say that $F$ and $G$ are \emph{parallel} if there exists $c>0$ such that
\begin{equation*}
  \dist(y,F) = \dist(x,G) = c
\end{equation*}
for all $x \in F$ and $y \in G$. Here $\dist(x,F) = \inf\{ d(x,z) : z \in F \}$. It is easy to see that if $F$ and $G$ are parallel, then $\dist(F,G) = d_H(F,G)$, where
\begin{equation*}
  \dist(F,G) = \inf\{ d(x,y) : x \in F \text{ and } y \in G \}
\end{equation*}
and
\begin{equation} \label{eq:Hausdorff_distance}
  d_H(F,G) = \sup(\{ \dist(y,F) : y \in G \} \cup \{ \dist(x,G) : x \in F \})
\end{equation}
is the \emph{Hausdorff distance}. Recall that $F$ is \emph{geodesic} if for all $x,y\in F$ there exists a path $\gamma\colon[0,1]\to F$ so that $\gamma(0)=x$, $\gamma(1)=y$ and $\length(\gamma)=d(x,y)$, where
\[
 \length(\gamma)=\sup\Bigl\{ \sum_{i=1}^n d(\gamma(t_{i-1}),\gamma(t_{i})) : n\in\N \text{ and } 0 = t_0 < t_1 < \cdots < t_{n-1} < t_n = 1 \Bigr\}.
\]
A metric space $(X,d)$ is called \emph{fibered} if there are an index set $I$ and closed sets $F_i\subset X$, $i\in I$, called \emph{fibers}, with $X=\bigcup_{i\in I}F_i$ so that the following properties are satisfied:
\begin{labeledlist}{F}
 \item \label{F1} Fibers are unbounded geodesic metric spaces: $(F_i,d|_{F_i})$ is an unbounded geodesic metric space for all $i\in I$.
 
 \item \label{F2} Fibers have positive distance: $\dist(F_i,F_j) > 0$ for all $i\neq j$.
 
 \item \label{F3} Non-parallel fibers diverge: $d_H(F_i,F_j)=\infty$ for all non-parallel $F_i$ and $F_j$.
 
 \item \label{F4} Parallel fibers are not isolated: for any fiber $F_i$, there exists a sequence of fibers $(F_{n})$ so that each $F_n$ is parallel to $F_i$ and $\dist( F_i , F_{n} ) \to 0$ as $n\to\infty$. 
\end{labeledlist}

The following lemma is an easy corollary of \cite[Theorem 1.1]{LeDonneXie2015}.

\begin{lemma} \label{lem:FiberedBiLipConstants}
  Suppose that $(X,d)$ and $(Y,\roo)$ are fibered spaces and $C\geq 1$. If $f \colon X \to Y$ is $(C,\eta)$-quasisymmetric such that it sends fibers of $X$ homeomorphically onto fibers of $Y$, then $f$ is $L$-bi-Lipschitz, where $L \ge 1$ depends only on $\eta$ and $C$.
\end{lemma} 

\begin{proof}
  By \cite[Theorem 1.1]{LeDonneXie2015}, there are constants $K \ge 1$ and $M > 0$ such that $K$ depends only on $\eta$ and $f$ satisfies
  \begin{equation} \label{eq:quasi-similarity}
    MK^{-1} d(x,y) \le \roo(f(x),f(y)) \le MK d(x,y)
  \end{equation}
  for all $x,y \in X$. Let $p$ and $w$ be the points of the condition \eqref{eq:C-condition}. Since \eqref{eq:quasi-similarity} and \eqref{eq:C-condition} give
  \begin{equation*}
    MK^{-1} \leq \frac{ \roo(f(p),f(w)) }{ d(p,w) } \le C
  \end{equation*}
  and similarly $C^{-1} \le MK$ we get $(C K)^{-1} \leq M\leq C K$. Thus $f$ satisfies
  \begin{equation*}
    C^{-1} K^{-2} d(x,y) \leq \roo(f(x),f(y)) \leq C K^2 d(x,y)
  \end{equation*}
  for all $x,y\in X$, which is what we wanted to show.
\end{proof}

In \S \ref{sec:horizontal-carpets}, we will need the lemma in the form where the spaces are unions of fibered spaces. The task is thus to show that each restriction of the $(C,\eta)$-quasisymmetric mapping to the fibered space is $(C',\eta)$-quasisymmetric for some $C'$ depending only on $\eta$ and $C$. The claim of the lemma then follows by applying Lemma \ref{lem:FiberedBiLipConstants} on each fibered space and gluing the obtained bi-Lipschitz mappings together.

\begin{lemma} \label{lem:manyfiberedcomponents}
  Suppose that $(X,d)$ and $(Y,\roo)$ are metric spaces such that $X=\bigcup_{i\in I} F_i$ and $Y=\bigcup_{j\in J} G_i$, where $F_i$ and $G_j$ are fibered spaces (in the restriction metrics). If $f\colon X\to Y$ is $(C,\eta)$-quasisymmetric such that it maps fibers of each $F_i$ homeomorphically onto fibers of some $G_j$, then $f$ is $L$-bi-Lipschitz, where $L$ depends only on $\eta$ and $C$.
\end{lemma}

\begin{proof}
  By the assumption, there exist points $p,w \in X$ so that $C^{-1} d(p,w)\leq \roo(f(p),f(w)) \leq C d(p,w)$. Let us first assume that $p,w\in F_i$ for some $i\in I$. Since $f|_{F_i}$ is $(C,\eta)$-quasisymmetric and fibers of $F_i$ are mapped homeomorphically onto fibers of $G_j$, we get, by Lemma \ref{lem:FiberedBiLipConstants}, that $f|_{F_i}$ is $L_1$-bi-Lipschitz, where $L_1$-depends only on $\eta$ and $C$. To show that also $f|_{F_{i'}}$ is bi-Lipschitz, we seek points $u,v\in F_{i'}$ so that $f|_{F_{i'}}$ satisfies condition \eqref{eq:C-condition} for $u,v$ and some $C'$ depending on only $\eta$ and $C$. Fix $x\in F_i$ and $u\in F_{i'}$. Since fibers of $F_i$ and $F_{i'}$ are unbounded geodesic spaces, we may choose $y\in F_i$ and $v\in F_{i'}$ so that $d(y,x)=d(u,x)=d(u,v)$. Now 
  \begin{align*}
    \frac{\roo(f(u),f(v))}{d(u,v)}
    &=\frac{ \roo(f(u),f(v)) }{ \roo(f(u),f(x)) } \frac{ \roo(f(u),f(x)) }{ \roo(f(y),f(x)) } \frac{\roo(f(y),f(x))}{d(y,x)} \frac{ d(y,x) }{ d(u,v) }\\
    &\leq \eta\biggl( \frac{ d(u,v) }{ d(u,x) } \biggr) \eta\biggl( \frac{ d(u,x) }{ d(y,x) } \biggr) L_1 = \eta(1)^2 L_1
  \end{align*}
  and, similarly,
  \begin{equation*}
    \frac{d(u,v)}{\roo(f(u),f(v))} \leq \eta(1)^2 L_1.
  \end{equation*}
  Thus $f|_{F_{i'}}$ is $(C',\eta)$-quasisymmetric where $C'=\eta(1)^2 L_1$ depends only on $\eta$ and $C$. Lemma \ref{lem:FiberedBiLipConstants} implies now that $f|_{F_{i'}}$ is $L_2$-bi-Lipschitz for some $L_2 \ge 1$ depending only on $\eta$ and $C$.

  Let us then assume that $p\in F_{i}$ and $w\in F_{i'}$ for some $i\neq i'$. Choose $v\in F_{i}$ so that $d(v,p)=d(p,w)$. Now
  \begin{align*}
  \frac{ \roo(f(v),f(p)) }{ d(v,p) }
  &= \frac{ \roo(f(v),f(p)) }{ \roo(f(w),f(p)) } \frac{\roo(f(w),f(p))}{d(p,w)} \frac{d(p,w)}{d(v,p)}
  \leq \eta\biggl( \frac{ d(v,p) }{ d(w,p) } \biggr) C
  = \eta(1) C
  \end{align*}
  and, similarly,
  \begin{equation*}
    \frac{d(v,p)}{\roo(f(v),f(p))} \leq \eta(1) C.
  \end{equation*}
  Thus $f|_{F_{i}}$ is $(C'',\eta)$-quasisymmetric where $C''=\eta(1)C$ depends only on $\eta$ and $C$. Now, continuing as in the first part of the proof, we conclude that there exists $L \ge 1$ depending only on $\eta$ and $C$ such that $f|_{F_i}$ is $L$-bi-Lipschitz for all $i \in I$.

  It remains to glue the mappings $f|_{F_i}$ together. Fix $x\in F_i$ and $y\in F_{i'}$. Choose $z\in F_i$ so that $d(x,y)=d(x,z)$. Now we have
  \begin{equation*}
    \frac{\roo( f(x) , f(y) )}{d(x,y)}
    = \frac{\roo( f(x) , f(y) )}{\roo( f(x) , f(z) )} \frac{\roo( f(x) , f(z) )}{d(x,z)} \frac{d(x,z)}{d(x,y)}
    \leq \eta\biggl( \frac{ d(x,y) }{ d(x,z) } \biggr) L =\eta(1)L
  \end{equation*}
  and, similarly,
  \begin{equation*}
    \frac{d(x,y)}{\roo( f(x) , f(y) )}\leq \eta(1)L.
  \end{equation*}
  Therefore $f$ is bi-Lipschitz with a constant $\max\{L,\eta(1)L\}$ which only depends on $\eta$ and $C$.
\end{proof}

\section{Horizontal self-affine carpets} \label{sec:horizontal-carpets}

In this section, we apply Theorem \ref{thm:quasilip} in a concrete setting. We consider horizontal self-affine carpets on the plane. We will first show, generalizing the result of Bandt and K\"aenm\"aki \cite[Theorem 1]{BandtKaenmaki2013}, that all weak tangents of such sets are unions of fibered spaces. We remark that \cite[Theorem 1]{BandtKaenmaki2013} is recently generalized in another direction by K\"aenm\"aki, Koivusalo, and Rossi; see \cite[Theorem 3.1]{KaenmakiKoivusaloRossi2015}. It is essential for us that we get the result for all weak tangents, not just for almost all tangents as in the above mentioned results. Recall that a set $E \subset \R^2$ is \emph{porous} with a constant $0<\alpha<1$ if for every $x \in E$ and $r>0$ there exists $y \in \R^2$ such that $B(y,\alpha r) \subset B(x,r)\setminus E$.

\begin{theorem} \label{thm:BK_general}
  If $E$ is a horizontal self-affine carpet, then weak tangents of $E$ are of the form
  \begin{equation*}
    (-\infty,w] \times \CL \cup [w,\infty) \times \CR,
  \end{equation*}
  where $\CL$ and $\CR$ are uniformly perfect porous sets, and at least one of them is nonempty.
\end{theorem}

Relying on this result, Theorem \ref{thm:quasilip} together with Lemma \ref{lem:manyfiberedcomponents} unveils that the class of quasisymmetric mappings between two horizontal self-affine carpets is rigid and the proof of Theorem \ref{thm:horizontal_carpet} follows. In particular, by Lemma \ref{lem:quasiliphausdorff}, such a quasisymmetric mapping can exist only when $E$ and $F$ have the same dimension. Let us now start adding details for these claims.

We will first define horizontal self-affine carpets. Let $\Phi=\{\fii_i\}_{i=1}^N$ be a collection of contractive self-maps on $\R^2$. The collection $\Phi$ is called an \emph{iterated function system (IFS)}. By Hutchinson \cite{Hutchinson1981}, there exists a unique non-empty compact set $E$, called the \emph{invariant set} of the iterated function system, satisfying
\[
  E = \bigcup_{i=1}^N \fii_i(E).
\]
Since we are interested in dimensional properties of $E$ we may assume that $\diam(E) = 1$. If the images $\fii_i(E)$ are pairwise disjoint, then the IFS is said to satisfy the \emph{strong separation condition (SSC)}. In the study of iterated function systems, it is often convenient to use the following notation. Let $N \ge 2$ be an integer and let $\Sigma_n = \{ 1,\ldots,N \}^n$ be the collection of all sequences of length $n$ formed from the set $\{ 1,\ldots,N \}$. If $\iii = (i_1,\ldots,i_n) \in \Sigma_n$, then we write $|\iii|=n$ and $\iii^- = (i_1,\ldots,i_{n-1})$. The set of all finite sequences $\bigcup_{n \in \N} \Sigma_n$ is denoted by $\Sigma_*$ and the set of all infinite sequences $\{ 1,\ldots,N \}^\N$ is denoted by $\Sigma$. If $\iii = (i_1,i_2,\ldots)$, then we write $\iii|_n = (i_1,\ldots,i_n) \in \Sigma_n$ for all $n \in \N$. The concatenation of two sequences $\iii$ and $\jjj$ is denoted by $\iii\jjj$. The correspondence between the invariant set $E$ and the set $\Sigma$ is given by the surjective mapping $\pi\colon \Sigma \to E$, defined by the relation
\[
  \{ \pi(\ii) \} = \bigcap_{n\in\N} \fii_{i_1} \circ \fii_{i_2} \circ\dots\circ \fii_{i_n}(K),
\]
where $K$ is any non-empty compact set satisfying $\bigcup_{i=1}^N \fii_i(K) \subset K$. If $\iii = (i_1,\ldots,i_n) \in \Sigma_n$ for some $n$, then we write $\fii_\iii = \fii_{i_1} \circ \dots \circ \fii_{i_n}$ and $E_\ii = \fii_{\ii}(E)$. The sets $E_\ii$, and the corresponding sets $[\ii] = \{ \jj\in\Sigma : \jjn{|\ii|} = \ii\}$, are the \emph{cylinder sets} of level $\ii$. 

We assume that all the mappings of the IFS are invertible and affine such that the linear parts are diagonal. Denoting $\fii_i(x_1,x_2) = (\alpha_1(i) x_1,\alpha_2(i) x_2) + (b_1(i),b_2(i))$ we thus have
\begin{align*}
  \ua &= \max\{ \alpha_1(i) : i \in \{ 1,\ldots,N\} \} < 1, \\
  \la &= \min\{ \alpha_2(i) : i \in \{ 1,\ldots,N\} \} > 0.
\end{align*}
In this case, the invariant set $E$ is called a \emph{self-affine carpet}. It is \emph{horizontal} if the corresponding IFS satisfies the SSC and the following two conditions hold:
\begin{labeledlist}{H}
 \item\label{H:horizontal} For each $i\in \{ 1,\ldots,N \}$ we have $\alpha_1(i) > \alpha_2(i)$.
 \item\label{H:projection} Every vertical line that intersects $X$, the convex hull of $E$, intersects $\fii_i(X)$ for at least two different $i\in \{ 1,\ldots,N \}$.
\end{labeledlist}
Observe that any Gatzouras-Lalley carpet satisfying the SSC and \ref{H:projection} is a horizontal self-affine carpet. By the iterative structure and compactness, \ref{H:projection} implies that each vertical line that intersects $X$, also intersects $E$. Thus the projection of $E$ onto the horizontal coordinate is a line segment. For a more detailed argument, see, for example, \cite[Remark 3.3]{KaenmakiKoivusaloRossi2015}. Furthermore, if
\begin{equation*} %\label{eq:hoo}
  \beta = \max\{ \alpha_2(i)/\alpha_1(i) : i \in \{ 1,\ldots,N \} \},
\end{equation*}
then the condition \ref{H:horizontal} implies that $0<\beta<1$. Finally, the SSC guarantees that
\begin{equation*}
  \delta = \min_{i\neq j}\{ \dist(\fii_i(E),\fii_j(E))\}>0.
\end{equation*}
To finish the preliminaries on horizontal self-affine carpets, let us introduce some more notation. Since the linear part of $\fii_i$ is $\diag(\alpha_1(i),\alpha_2(i))$ we see that the linear part of $\fii_\iii$ is $\diag(\alpha_1(\iii),\alpha_2(\iii))$, where $\alpha_j(\iii) = \alpha_j(i_1) \cdots \alpha_j(i_n)$ for all $\iii = (i_1,\ldots,i_n) \in \Sigma_n$. We define
\begin{align*}
  n(\iii,t) &= \max\{ n \in \N : E_\jjj \cap B(\pi(\iii),t) = \emptyset \text{ for all } \jjj \in \Sigma_n \setminus \{ \iii|_n \} \}, \\
  n^*(t) &= \min\{ n \in \N : \ualpha^n < t \}, \\
  n_*(t) &= \max\{ n \in \N : \lalpha^n \delta > t \},
\end{align*}
for all $\ii\in \Sigma$ and $0<t<1$. The existence of $n(\iii,t)$ is guaranteed by the SSC. It is easy to see that $n(\ii,t)$, $n^*(t)$, and $n_*(t)$ increase as $t$ decreases to zero. We abbreviate $\iin{n(\ii,t)}$ by $\iin{t}$.

\begin{lemma} \label{lem:littlethings}
  For every $\iii \in \Sigma$ and $0<t<1$ it holds that
  \begin{enumerate}
    \item\label{nstar} $n_{*}(t) \leq n(\ii,t) \leq n^{*}(t)$,
    \item\label{ratio} $\la t \leq \alpha_2(\iin{t})  \leq \delta^{-1}t$.
  \end{enumerate}
\end{lemma}

\begin{proof}
  \eqref{nstar} If $\ua^{n}<t$, then $\fii_{\iin{n}}(E)\subset B(\pi(\ii),t)$. Thus $B(\pi(\iii),t)$ contains two cylinders of level $n+1$ and therefore, $n(\ii,t) \leq n$. To show the other inequality, observe first that the distances between cylinders of level $n+1$ are at least $\la^{n}\delta$. If $\la^{n}\delta>t$, then $B(\pi(\ii),t)$ intersects $E_{\iin{n}i}$ for some $i\in\{1,\ldots,N\}$ but it can not intersect another cylinder of level $n+1$. Thus $n \leq n(\ii,t)$.
 
  \eqref{ratio} Observe first that we have $\la\alpha_2(\iin{t}^-) \leq \alpha_2(\iin{t})$. Since any vertical line through $E_{\iin{t}}$ intersects at least two of its sub-cylinders we have $t<\alpha_2(\iin{t}^-)$ yielding the first inequality. Since $B(\pi(\ii),t)$ intersects at least two sub-cylinders of $E_{\iin{t}}$ we have $\alpha_2(\iin{t})\delta<t$. Thus also the second inequality holds.
\end{proof}

Let $Q=[h,h']\times[v,v']$ be the smallest closed rectangle containing $E$ with sides parallel to the coordinate axis. Observe that, since we assumed $\diam (E) = 1$, the height of $Q$ is at most $1$. We denote $Q_\ii=\fii_\ii(Q)$ and call $Q_\ii$ a \emph{construction rectangle} of level $|\ii|$. By \emph{horizontal endings} of a rectangle $Q_\ii$ we mean the vertical line segments $\fii_\ii( \{h\}\times [v,v'])$ and $\fii_\ii( \{h'\}\times [v,v'])$. Note that even though the SSC means that sets $\fii_\iii(E)$ and $\fii_\jj(E)$ are disjoint for $\iii$ and $\jjj$ with $[\iii] \cap [\jjj] = \emptyset$, it does not imply that $Q_\iii$ and $Q_\jjj$ are disjoint. This is not a problem since the crucial thing is that $Q_\iii$ approximates $\fii_\ii(E)$ well.

\begin{lemma} \label{lem:endings}
  For every $K\in\N$ there exists $t_K>0$ such that for any $0<t<t_K$ and $\iii \in \Sigma$ the ball $B(\pi(\iii),t)$ intersects at most one vertical line containing a horizontal ending of some rectangle $Q_{\iin{t}\jj}$ with $|\jj|=K$.
\end{lemma}

\begin{proof}
  Fix $K \in \N$. Observe first that the construction rectangles of level $K$ have $2 N^K$ horizontal endings. Denote the smallest positive horizontal distance between them by $\delta_K>0$ and choose $0<t_K<1$ so that
  \begin{equation*}
    \delta_K \la \beta^{-n_*(t)} \geq 3
  \end{equation*}
  for all $0<t<t_K$. Fix $\iii \in \Sigma$ and $0<t<t_K$. Since $\alpha_2(\iin{t})/\alpha_1(\iin{t}) \leq \beta^{n(\ii,t)}$ Lemma \ref{lem:littlethings} gives
  \begin{equation*}
    \delta_K\alpha_1(\iii|_t) \ge \delta_K\alpha_2(\iii|_t)\beta^{-n(\iii,t)} \ge \delta_K\alpha_2(\iii|_t)\beta^{-n_*(t)} \ge \delta_K\lalpha \beta^{n_*(t)}t \ge 3t.
  \end{equation*}  
  Observe that the level $n(\ii,t)+K$ horizontal endings in $Q_{\iin{t}}$ have horizontal separation either zero or at least $\delta_K\alpha_1(\iii|_t)$. Therefore we conclude that $B(x,t)$ can intersect at most one vertical line containing such a ending.
\end{proof}

If $E \subset \R^2$, then the \emph{vertical slice} of $E$ at $y=(y_1,y_2) \in \R^2$ is
\[
  V_y(E) = \{ x_2 \in \R : (x_1,x_2) \in E \text{ and } x_1=y_1 \}.
\]
Note that $\{ y_1 \} \times V_y(E) = E \cap \{ (y_1,z) \in \R^2 : z \in \R \}$. Since $V_y(E)$ does not depend on $y_2$ we denote it also by $V_{y_1}(E)$.

\begin{lemma} \label{lem:verticalsilices}
  If $E$ is a horizontal self-affine carpet, then every vertical slice of $E$ is porous with a constant $\min\{ \delta, 1 \}/4$ and uniformly perfect with a constant $\delta^{-1} \lalpha^{-k-1}$, where $k$ is the smallest integer with $\ualpha^k<\delta$.
\end{lemma}

\begin{proof}
  Let us first show that the vertical slices are porous. Let $z\in E$ and fix $t>0$. Let $y=(y_1,y_2)=\pi(\ii)$ be so that $y_1=z_1$ and $y_2\in V_z(E)$. If $\alpha_2(\iin{t}) < t/2$, then by the definition of $n(\ii,t)$ we have
  \[
    V_z(E)\cap (y_2+t/2,y_2+t)=\emptyset
  \]
  and $(y_2+t/2,y_2+t) \subset (y_2-t,y_2+t)$. Furthermore, if $\alpha_2(\iin{t})\geq t/2$, then by the SSC, the distances of level $n(\ii,t)+1$ cylinder sets are at least $\delta\alpha_2(\iin{t}) \geq \delta t/2$. Thus there exists $x_2$ so that
  \[
    V_z(E) \cap (x_2-\delta t/4,x_2+\delta t/4) = \emptyset
  \]
  and $(x_2-\delta t/4,x_2+\delta t/4) \subset (y_2-t,y_2+t)$. Thus $V_z(E)$ is porous with the constant $\min\{ \delta,1 \}/4$.

  Let us then show the uniform perfectness. Let $z\in E$ and fix $y=(y_1,y_2)=\pi(\ii)$ so that $y_1=z_1$ and $y_2\in V_z(E)$. Let $t>0$ be so that
  \[
    V_z(E)\setminus (y_2-t,y_2+t)\neq \emptyset.
  \]
  Let $k$ be the smallest integer with $\ualpha^k < \delta$ and let $\jjj \in \Sigma_k$ be such that $y \in E_{\iii|_t \jjj}$. By the condition \ref{H:projection}, a vertical line through $y$ intersects two sub-cylinders of $E_{\iin{t}\jj}$. Thus there exists a point $x=(y_1,x_2)\in E_{\iin{t}\jj}$ so that it is not contained in the same $n(y,t)+k+1$ level cylinder with $y$. Now, by Lemma \ref{lem:littlethings}\eqref{ratio} and the choice of $k$, we have
  \[
    |x - y| \leq \alpha_2(\iin{t})\ualpha^k \leq \delta^{-1}t\ualpha^k < t
  \]
  and
  \[
    |x - y| \geq \delta \alpha_2(\iin{t})\lalpha^{k} \geq \delta \lalpha^{k+1} t.
  \]
  In other words, $x\in V_z(E) \cap (y_2-t,y_2+t)\setminus (y_2 - \delta \lalpha^{k+1} t, y_2 - \delta \lalpha^{k+1} t)$ and $V_z(E)$ is uniformly perfect with the constant $\delta^{-1} \lalpha^{-k-1}$.
\end{proof}

Let $E \subset \R^2$ be closed, $p\in E$, and $t>0$. Note that the pointed metric space $(E,|\cdot|/t,p)$ is homothetic to $((E-p)/t,|\cdot|,0)$ via the homothety $x \mapsto (x-p)/t$. Here $(E-p)/t = \{ (x-p)/t \in \R^2 : x \in E \}$. Therefore, whenever we consider weak tangents of subsets of $\R^2$, we can always choose the associated bi-Lipschitz embeddings to be this homothety. Recall that the weak tangents are unique up to an isometry. We say that $T\subset\R$ is a \emph{weak vertical slice tangent} of $E$ if there exists a sequence $(y_i)$ of points in $E$ and a sequence $(t_i)$ of positive reals converging to zero such that $((V_{y_i}(E)-\proj_2y_i)/t_i,|\cdot|,0)$ converges to $(T,|\cdot|,0)$. Here $\proj_2$ is the orthogonal projection onto the vertical axis. We make the corresponding choice for the embeddings also in this case.

If $D,F,G \subset \R^2$, then we write
\begin{equation*}
  d_H^D(F,G) = d_H(F \cap D,G \cap D),
\end{equation*}
where $d_H$ is the Hausdorff distance defined in \eqref{eq:Hausdorff_distance}.
% Observe that if $(F_i)$ is a sequence of non-empty closed sets in $\R^2$, then $F_i$ converges to a non-empty closed set $F \subset \R^2$ in the sense of \S\ref{sec:MappingPackages} if
% \begin{equation*}
%   \lim_{i \to \infty} d_H^{B(0,R)}(F_i,F) = 0
% \end{equation*}
% for all $R > 0$.
Furthermore, if $\II$ is a collection of sets, then we will slightly abuse notation and write $\II$ to denote also the union $\bigcup_{I \in \II} I$.

\begin{lemma} \label{lem:epspatterns}
  For every $K\in\N$ there exists $t_K>0$ such that for any $0<t<t_K$ and $x \in E$ there are $u, w ,v \in \R$ so that
  \[
    d_H^{B(x,t)}( (-\infty,w] \times V_{u}(E) \cup [w,\infty) \times V_{v}(E), E) \leq t \delta^{-1}\ualpha^K.
  \]
\end{lemma}

\begin{proof}
  Fix $K\in\N$, let $t_K>0$ be as in Lemma \ref{lem:endings}, and choose $0<t<t_K$. Let $\iii \in \Sigma$ be such that $x=\pi(\iii)$. If $|\jj|=K$, then, by Lemma \ref{lem:littlethings}\eqref{ratio}, the height of $Q_{\iin{t}\jj}$ is
  \[
    \alpha_2(\iin{t}\jj) = \alpha_2(\iin{t})\alpha_2(\jj) \leq \delta^{-1}t\ualpha^K.
  \]
  By the condition \ref{H:projection}, we have $d_H( E_{\iin{t}\jj} , Q_{\iin{t}\jj} ) \leq \delta^{-1}t\ualpha^K$. Furthermore, by Lemma \ref{lem:endings}, there is a point $w\in\R$ such that all the horizontal endings of rectangles $Q_{\iin{t}\jj}$ intersecting $B(x,t)$ are contained in the line $\{ w \} \times \R$. Define
  \begin{align*}
    \II &= \{\proj_{2} Q_{\iin{t}\jj} : |\jj| = K \text{ and } B(x,t)\cap Q_{\iin{t}\jj} \cap (-\infty , w) \times \R  \neq \emptyset \}, \\
    \JJ &= \{\proj_{2} Q_{\iin{t}\jj} : |\jj| = K \text{ and } B(x,t)\cap Q_{\iin{t}\jj} \cap (w,\infty) \times \R \neq \emptyset \}.
  \end{align*}
  This immediately means that
  \begin{equation*}
    E \cap B(x,t) \subset ((-\infty,w] \times \II \cup [w,\infty) \times \JJ) \cap B(x,t) \subset E(\delta^{-1}t\ualpha^K) \cap B(x,t),
  \end{equation*}
  where $E(\eps)$ is the $\eps$-neighborhood of $E$. Now fix $u$ and $v$ so that $w-t<u<w<v<w+t$ and consider the vertical slices $V_{u}(E)$ and $V_{v}(E)$. Recall that the vertical lines containing the horizontal endings of the rectangles $Q_{\iin{t}\jj}$ are at least $3t$ apart. Therefore, in the above inequalities $E$ can be replaced by the set
  \[
    (-\infty,w] \times V_{u}(E) \cup [w,\infty) \times V_{v}(E).
  \]
  Thus
  \[
    d_H^{B(x,t)}( (-\infty,w] \times V_{u}(E) \cup [w,\infty) \times V_{v}(E), E) \leq t \delta^{-1}\ualpha^K.
  \]
  This proves the claim in the case rectangles $Q_{\iin{t}\jj}$ have horizontal endings intersecting $B(x,t)$. If there are no such endings, then we may choose $u=v=w=\proj_1(x)$, where $\proj_1$ is the orthogonal projection onto the horizontal axis.
\end{proof}

We are now ready to prove Theorem \ref{thm:BK_general}.

\begin{proof}[Proof of Theorem \ref{thm:BK_general}]
  Let $W$ be a weak tangent of $E$. This means that there exist a sequence $(x_i)$ of points in $E$ and a sequence $(t_i)$ of positive reals converging to zero such that
  \begin{equation}
  \label{eq:covergence-W}
  \begin{split}
   &\lim_{i \to \infty} \sup\{ \dist(x,W) : x \in (E-x_i)/t_i \cap B(0,R) \} = 0\text{ and }\\
   &\lim_{i \to \infty} \sup\{ \dist(y,(E-x_i)/t_i) : y \in W \cap B(0,R) \} = 0
  \end{split}
  \end{equation}
  for all $R>0$. Fix $R>0$, let $K \in \N$, and choose $i_0 \in \N$ such that $0<t_i<t_K/2R$ for all $i \ge i_0$, where $t_K>0$ is as in Lemma \ref{lem:epspatterns}. Now for each $i \ge i_0$, by Lemma \ref{lem:epspatterns}, there exist $u_i,w_i,v_i \in \R$ so that
  \begin{equation*}
    d_H^{B(x_i,2Rt_i)}\bigl((-\infty,w_i]\times V_{u_i}(E) \cup [w_i,\infty)\times V_{v_i}(E),E\bigr) \le 2Rt_i\delta^{-1}\ualpha^K.
  \end{equation*}
  Therefore,
  \begin{equation*}
    d_H^{B(0,2R)}\bigl(\bigl(((-\infty,w_i]\times V_{u_i}(E) \cup [w_i,\infty)\times V_{v_i}(E))-x_i\bigr)/t_i,(E-x_i)/t_i\bigr) \le 2R\delta^{-1}\ualpha^K
  \end{equation*}
  for all $i \ge i_0$. The use of double radius here ensures that we do not need to worry about the convergence on the boundary of $B(0,R)$. As $i$ increases we can let $K \to \infty$, and so the above Hausdorff distance converges to zero. Thus by \eqref{eq:covergence-W},
  \begin{equation*}
    \bigl(((-\infty,w_i]\times V_{u_i}(E) \cup [w_i,\infty)\times V_{v_i}(E))-x_i\bigr)/t_i
  \end{equation*}
  converges to $W$. Due to the convergence, there exists $w \in \R$ such that
  \begin{equation*}
    W = (-\infty,w]\times\CL \cup [w,\infty)\times\CR,
  \end{equation*}
  where $\CL$ and $\CR$ are weak vertical slice tangents of $E$. By Lemma \ref{lem:verticalsilices}, vertical slices of $E$ are porous and uniformly perfect. Since these properties are preserved in the limit we have finished the proof.
\end{proof}

The following example illustrates that $\CL$ and $\CR$ can be disjoint.

\begin{example} \label{ex:CLCR}
  Let $g \colon \R^2 \to \R^2$, $g(x_1,x_2)=(0.5x_1,0.2x_2)$, and then set $f_1=g$, $f_2=g+(0.5,0.25)$, $f_3=g+(0,0.55)$, and $f_4=g+(0.5,0.8)$. The invariant set $E$ of the iterated function system $\{f_i\}_{i=1}^4$ is depicted in Figure \ref{fig:exampleCLCR}. Notice that the vertical center line of the unit cube contains left and right endings of the construction rectangles. To obtain a weak tangent so that $\CL$ and $\CR$ are disjoint and nonempty, we just choose the defining sequences $(x_i)$ and $(t_i)$ so that for each $i$ the vertical center line of an appropriate construction rectangle is in the middle of the ball $B(x_i,t_i)$. This ensures that the unit ball of the weak tangent also has a separating vertical line in the middle.
\end{example}

%%%%%%%%%%%%%%%%%%%%%%%%%%%%%%%%%%%%%%%%%%%%%%%%%%%%%%%%%%%%%%%TIKZ-PICTURE%%%%%%%%%%%%%%%%%%%%%%%%%%%%%%%%%%%%
  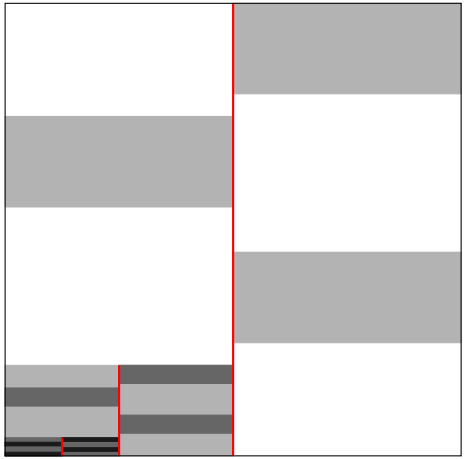
\begin{figure}[t]
    \begin{centering}
      \begin{tikzpicture}[scale=0.3]
        \filldraw[black!30!white] (0,0) rectangle (10,4); %these are the first level cylinders
        \filldraw[black!30!white] (10,5) rectangle (20,9);
        \filldraw[black!30!white] (0,11) rectangle (10,15);
        \filldraw[black!30!white] (10,16) rectangle (20,20);

        \filldraw[black!60!white] (0,0) rectangle (0.5*10,0.2*4); %lets draw a second level cylinder
        \filldraw[black!60!white] (0.5*10,0.2*5) rectangle (0.5*20,0.2*9);
        \filldraw[black!60!white] (0.5*0,0.2*11) rectangle (0.5*10,0.2*15);
        \filldraw[black!60!white] (0.5*10,0.2*16) rectangle (0.5*20,0.2*20);

        \filldraw[black!90!white] (0,0) rectangle (0.5*0.5*10,0.2*0.2*4); %lets draw a third level cylinder
        \filldraw[black!90!white] (0.5*0.5*10,0.2*0.2*5) rectangle (0.5*0.5*20,0.2*0.2*9);
        \filldraw[black!90!white] (0.5*0.5*0,0.2*0.2*11) rectangle (0.5*0.5*10,0.2*0.2*15);
        \filldraw[black!90!white] (0.5*0.5*10,0.2*0.2*16) rectangle (0.5*0.5*20,0.2*0.2*20);

        \draw[red,thick] (10,0) -- (10,20); %the centre line of the unit cube
        \draw[red,thick] (5,0) -- (5,4); %the centre line of the cylinder of level 2
        \draw[red,thick] (2.5,0) -- (2.5,0.2*4); %the centre line of the cylinder of level 3

        \draw (0,0) rectangle (20,20); %The unite square
      \end{tikzpicture}
    \end{centering}
    \caption{The picture depicts some of the construction rectangles of the set $E$ in Example \ref{ex:CLCR}. The darkness of the color indicates the level of the construction rectangle and the red lines represent the vertical center lines.}
    \label{fig:exampleCLCR}
  \end{figure}
%%%%%%%%%%%%%%%%%%%%%%%%%%%%%%%%%%%%%%%%%%%%%%%%%%%%%%%%%%%%%%%%%%%%%%%%%%%%%%%%%%%%%%%%%%%%%%%%%%%%%%%%%%%%%%%

Relying on Theorem \ref{thm:BK_general}, we can now prove Theorem \ref{thm:horizontal_carpet}. The proof follows from Theorem \ref{thm:quasilip} and Lemma \ref{lem:manyfiberedcomponents} by showing that quasisymmetric weak tangent mappings act between unions of fibered spaces.

\begin{proof}[Proof of Theorem \ref{thm:horizontal_carpet}]
  Let $f \colon E \to F$ be an $\eta$-quasisymmetric mapping. By Theorem \ref{thm:quasilip}, it suffices to show that any $(1,\eta)$-quasisymmetric weak tangent mapping $\hat f$ of $f$ is $L$-bi-Lipschitz with $L$ depending only on $\eta$. By Theorem \ref{thm:BK_general}, the domain and codomain of $\hat f$ are $(-\infty,w] \times \CL \cup [w,\infty) \times \CR$ and $(-\infty,w'] \times \CL' \cup [w',\infty) \times \CR'$, respectively, where $\CL$, $\CR$, $\CL'$, and $\CR'$ are uniformly perfect porous sets and $w,w' \in \R$. Define
  \begin{align*}
    \TL &= \CL \setminus \CR, \\
    \TR &= \CR \setminus \CL, \\
    T &= \CL \cap \CR,
  \end{align*}
  and likewise $\TL'$, $\TR'$, and $T'$ in the codomain side. For topological reasons, it is clear that $\hat f(\R \times T) = \R \times T'$ and $\hat f((-\infty,w] \times \TL \cup [w,\infty) \times \TR) = (-\infty,w'] \times \TL' \cup [w',\infty) \times \TR'$. We will show that $(-\infty,w] \times \TL \cup [w,\infty) \times \TR$ and $(-\infty,w'] \times \TL' \cup [w',\infty) \times \TR'$ are fibered spaces and that $\R \times T$ and $\R \times T'$ both contain a suitable fibered space. It suffices to show the claim for $(-\infty,w] \times \TL \cup [w,\infty) \times \TR$ and $\R \times T$ since the same proof applies in the codomain side.
  
  Let us first show that $(-\infty,w] \times \TL \cup [w,\infty) \times \TR$ is a fibered space. Since $\TL$ and $\TR$ are disjoint we may define
  \begin{equation*}
    F_y =
    \begin{cases}
      (-\infty,w] \times \{ y \}, &\text{if } y \in \TL, \\
      [w,\infty) \times \{ y \}, &\text{if } y \in \TR.
    \end{cases}
  \end{equation*}
  Now clearly
  \begin{equation*}
    \bigcup_{y \in \TL \cup \TR} F_y = (-\infty,w] \times \TL \cup [w,\infty) \times \TR
  \end{equation*}
  and the fibers $F_y$ are unbounded geodesic metric spaces. Thus the condition \ref{F1} is satisfied. Since
  \begin{equation*}
    \dist(F_y,F_z) = |y-z| > 0
  \end{equation*}
  for all $y,z \in \TL \cup \TR$ with $y \ne z$, also the condition \ref{F2} is satisfied. Fibers $F_y$ and $F_z$ are non-parallel only if $y \in \TL$ and $z \in \TR$ or vice versa. This implies that $d_H(F_y,F_z) = \infty$ and the condition \ref{F3} is satisfied. To check \ref{F4}, it suffices to show that there are no isolated points in $\TL$ (nor in $\TR$). Recall that $\CL$ and $\CR$ are closed sets with no isolated points. Hence for each $y \in \TL$ there exists a sequence $(y_i)$ in $\CL$ which converges to $y$. Suppose to the contrary that $y$ is an isolated point of $\TL$. By the definition of $\TL$ this means that $y_i \in \CR$ for all large enough $i$. Since $\CR$ is closed we conclude that $y \in \CR$ and hence $y \notin \TL$ which is a contradiction.
  
  Let us then show that $\R \times S$ is a fibered space, where
  \begin{equation*}
    S = T \setminus \overline{\TL \cup \TR}.
  \end{equation*}
  Let $S'$ be the corresponding set in the codomain side. Observe that $S \cup \TL \cup \TR$ is dense in $T \cup \TL \cup \TR = \CL \cup \CR$. As above, the conditions \ref{F1}--\ref{F3} follow immediately. To check \ref{F4}, it suffices to show that there are no isolated points in $S$. Suppose to the contrary that $y \in S$ is an isolated point of $S$. Recalling that $\CL \cup \CR$ has no isolated points, there exists a sequence $(y_i)$ in $T \cup \TL \cup \TR$ which converges to $y$. Since $y$ is an isolated point and, by the definition of $S$, there does not exist a sequence in $\TL \cup \TR$ converging to $y$, we have $y_i \in T \setminus S$ for all large enough $i$. We may assume that $y_i \in T \setminus S$ for all $i$. For each $y_i$, by the definition of $S$, there exists a sequence $(z_j^i)_j$ in $\TL \cup \TR$ converging to $y_i$. Therefore, by choosing suitable points from these sequences, we may construct a sequence $(z_{j_i}^i)_i$ in $\TL \cup \TR$ which converges to $y$. Thus $y \notin S$ which is a contradiction.
  
  To finish the proof, we use the denseness and apply Lemma \ref{lem:manyfiberedcomponents} in the union of these fibered spaces. Since $\hat f(\R \times T) = \R \times T'$ it follows from the continuity of $\hat f$ that $\hat f(\R \times S) = \R \times S'$. Therefore $\hat f$ maps the fibers of $\R \times S$ homeomorphically onto the fibers of $\R \times S'$ and the fibers of $(-\infty,w] \times \TL \cup [w,\infty) \times \TR$ homeomorphically onto the fibers of $(-\infty,w'] \times \TL' \cup [w',\infty) \times \TR'$. Since $\R \times S \cup (-\infty,w] \times \TL \cup [w,\infty) \times \TR$ is dense in $(-\infty,w] \times \CL \cup [w,\infty) \times \CR$ and $\hat f$ is $(1,\eta)$-quasisymmetric, we have, by the continuity of $\hat f$, that $\hat f$ restricted to $\R \times S \cup (-\infty,w] \times \TL \cup [w,\infty) \times \TR$ is $(2,\eta)$-quasisymmetric. By Lemma \ref{lem:manyfiberedcomponents}, this restriction is $L$-bi-Lipschitz, where $L$ depends only on $\eta$. Since $\hat f$ is $L$-bi-Lipschitz on a dense set it is $L$-bi-Lipschitz on the whole set.
\end{proof}

We will then turn to the proof of Theorem \ref{thm:conformal_dimensio}. Let $Q=[0,1]^2$. Recall that a set $M \subset Q$ is a \emph{Furstenberg miniset} of $E \subset Q$ if $M \subset (\lambda E+z) \cap Q = \{ \lambda x+z : x \in E \} \cap Q$ for some $\lambda \ge 1$ and $z \in \R^2$. The number $\lambda$ is called the \emph{scaling coefficient} of the miniset $M$. A set $M \subset Q$ is a \emph{Furstenberg microset} of a compact set $E \subset Q$ if there exists a sequence $(M_n)$ of Furstenberg minisets of $E$ such that $d_H(M_n,M) \to 0$ as $n \to \infty$. The sequence $(\lambda_n)_{n \in \N}$, where each $\lambda_n$ is a scaling coefficient of the miniset $M_n$, is called the \emph{scaling sequence} of the microset $M$. Note that a Furstenberg miniset is clearly a Furstenberg microset. Furthermore, if the scaling sequence of a microset $M$ is unbounded, then $M$ is a subset of a weak tangent of $E$.

\begin{proposition} \label{prop:microAssouad}
  If $E \subset Q$ is compact, then there exists a Furstenberg microset $M$ of $E$ having unbounded scaling sequence such that $\dimh M \geq \dima E$.
\end{proposition}

\begin{proof}
  Bishop and Peres \cite[Lemma 2.4.4]{BishopPeres2016} have shown that for any compact set $K \subset [0,1]$ there exists a Furstenberg microset $M$ of $K$ such that $\dimh M \ge \udimm K$. Although our proof here is a modification of their argument, we give full details for the convenience of the reader.

  Let $\tilde N_n(A)$ be the number of $n$ level dyadic cubes that intersect $A$ and define
  \[
    N_n(E) = \max\{\tilde N_n(A): A \text{ is a Furstenberg microset of } E \}.
  \]
  By \cite[Proposition 3.13]{KaenmakiRossi2015}, we have $\dima E \leq \lim_{n \to \infty} \log N_n(E)/\log 2^n$. Denote this limit by $t$. Let $M_n$ be a microset of $E$ such that $N_n(E) = \tilde N_n(M_n)$ for all $n \in \N$ and choose a probability measure $\nu_n$ on $M_n$ such that $\nu_n(D) = N_n(E)^{-1}$ for all $n$-level dyadic cubes $D$ that intersect $M_n$. For each dyadic cube $D$, let $S^D$ be the homotethy that maps $D$ to $Q$. Note that $S^D(x) = \lambda(x-a)$ for some $\lambda > 0$ and $a \in \R^2$. We set $A^D = S^D(A \cap D)$ for all compact sets $A \subset Q$ and $\nu^D = \nu(D)^{-1} S^D\nu$ for all probability measures $\nu$. Here $S^D\nu$ is the pushforward measure of $\nu$ under $S^D$.
  
  Fix $n \in \N$ and $0<s<t$. We will show that there exist $k(n,s) \in \N$ and a dyadic cube $D_{n,s}$ of level $l_n \ge n$ so that
  \begin{equation} \label{eq:F-claim}
    \frac{\nu_{k(n,s)}(D')}{\nu_{k(n,s)}(D_{n,s})} \le \frac{\diam(D')^s}{\diam(D_{n,s})^s}
  \end{equation}
  for all $l'$-level dyadic cubes $D' \subset D_{n,s}$ for all $l_n < l' \le l_n+n$. Here $\diam(A) = \sup\{ |x-y| : x,y \in A \}$ is the \emph{diameter} of the set $A$. The claim means that when we look at most $n$ levels further, the measure appears to be quite evenly spread in $Q_{n,s}$.

  To show \eqref{eq:F-claim}, we argue by contradiction. So, if the claim does not hold, then for every $k,n \in \N$ and an $n$-level dyadic cube $D_n$ of positive measure there exists a dyadic cube $D_{l_n^1}\subset D_n$ of level $n < l_n^1 \le 2n$ so that
  \begin{equation*}
    \frac{\nu_k(D_{l_n^1})}{\nu_k(D_n)} > \frac{\diam(D_{l_n^1})^s}{\diam(D_n)^s}.
  \end{equation*}
  We choose $D_n$ such that $\nu_k(D_n) \ge 2^{-2n}$. This can be done since $\nu_k$ is a probability measure and there are $2^{2n}$ dyadic cubes of level $n$. Since the claim fails also for $D_{l_n^1}$ we find a dyadic cube $D_{l_n^2}$ of level $l_n^1 < l_n^2 \le l_n^1+n$ so that
  \begin{equation*}
    \frac{\nu_k(D_{l_n^2})}{\nu_k(D_{l_n^1})} > \frac{\diam(D_{l_n^2})^s}{\diam(D_{l_n^1})^s}.
  \end{equation*}
  Continuing in this manner let $m$ be the largest integer for which $\diam(D_{l_n^m}) > 2^{-k}\sqrt{2}$. Thus $k-n \le l_n^m < k$. By telescoping, we get
  \begin{equation} \label{eq:F-arvio1}
    \nu_k(D_{l_n^m}) > \frac{\diam(D_{l_n^m})^s}{\diam(D_n)^s} \nu_k(D_n) = 2^{-l_n^ms} 2^{ns} \nu_k(D_n) \ge 2^{-l_n^ms} 2^{ns} 2^{-2n}.
  \end{equation}
  Let $s < s' < t$ and observe that there exists $k_0 \in \N$ such that $N_k(E) > 2^{ks'}$ for all $k \ge k_0$. Since $D_{l_n^m}$ contains at most $2^{2n}$ dyadic cubes of level $k$ we have
  \begin{equation} \label{eq:F-arvio2}
    \nu_k(D_{l_n^m}) \le \frac{2^{2n}}{N_k(E)} \le 2^{2n}2^{-ks'}
  \end{equation}
  for all $k \ge k_0$. Putting \eqref{eq:F-arvio1} and \eqref{eq:F-arvio2} together gives
  \begin{equation*}
    2^{-ks} 2^{ns} 2^{-2n} < 2^{2n}2^{-ks'}
  \end{equation*}
  which is clearly a contradiction if $k$ is chosen to be large enough. Therefore \eqref{eq:F-claim} holds.

  Let us now use \eqref{eq:F-claim} to prove the proposition. For every $n \in \N$ we choose $0<s_n<t$ such that $s_n \to t$ as $n \to \infty$. We consider the microsets $K_n = (M_{k(n,s_n)})^{D_{n,s_n}}$ and probability measures $\mu_n = (\nu_{k(n,s_n)})^{D_{n,s_n}}$. Let $K$ and $\mu$ be such that $K_n \to K$ in the Hausdorff distance and $\mu_n \to \mu$ weakly along some subsequence (which we keep denoting as the original sequence). Observe that $\spt\mu \subset K$ and that $K$, as a limit of microsets, is a microset of $E$. Furthermore, if $U(x,r)$ is an open ball centered at $x \in K$ with radius $r>0$, then
  \begin{equation*}
    \mu(U(x,r)) \le \liminf_{n \to \infty} \mu_n(U(x,r)).
  \end{equation*}
  Notice also that there is a constant $p$, depending only on the dimension of the ambient space, such that any ball of radius $r$ can be covered by $p$ many dyadic cubes $D_i$ with $r/2 < \diam(D_i) \le r$. By \eqref{eq:F-claim}, we have $\mu_n(D_i) = \nu_{k(n,s_n)}(D_i)^{-1}S^{D_{n,s_n}}\nu_{k(n,s_n)}(D_i) \le r^{s_n}$ whenever $2^{-n} < r < 1$. Thus we have
  \begin{equation*}
    \mu(U(x,r)) \le \liminf_{n \to \infty} \mu_n(U(x,r)) \le \liminf_{n \to \infty} \sum_{i=1}^p \mu_n(D_i) \le \liminf_{n \to \infty} \sum_{i=1}^p r^{s_n} = pr^t
  \end{equation*}
  for all $x \in K$ and $0<r<1$. By the definition of the Hausdorff measure, this implies that $\dimh(K) \ge t$. The proof is finished.
\end{proof}

The following proposition is a rather immediate corollary of Proposition \ref{prop:microAssouad}.

\begin{proposition} \label{cor:weakdimension}
  If $E\subset Q$ is compact, then there exists a weak tangent $W$ of $E$ such that
  \[
    \dimh W \cap Q = \dima W \cap Q = \dima E.
  \]
\end{proposition}

\begin{proof}
  By Proposition \ref{prop:microAssouad}, there exists a Furstenberg microset $M$ of $E$ having unbounded scaling sequence such that $\dimh M \geq \dima E$. Hence, $M$ is a subset of $W\cap Q$, where $W$ is a weak tangent of $E$. By \cite[Proposition 2.1]{Mackay2011}, we have $\dima W \cap Q \le \dima E$. Therefore,
  \[
    \dimh M \leq \dimh W \cap Q \leq \dima W \cap Q \leq \dima E \leq \dimh M
  \]
  which is what we wanted to show.
\end{proof}

We remark that in Proposition \ref{cor:weakdimension} it is essential to use weak tangents -- there does not necessarily exist such tangent sets; see \cite[Example 2.20]{LeDonneRajala2015}. Together with Remark \ref{rem:weak-tangent}, this observation further emphasizes the use of weak tangents in our analysis. We are now ready to prove Theorem \ref{thm:conformal_dimensio}.

\begin{proof}[Proof of Theorem \ref{thm:conformal_dimensio}]
  Let $E$ be a horizontal self-affine carpet. By Proposition \ref{cor:weakdimension}, there exists a weak tangent $W$ of $E$ such that
  \begin{equation*}
    \dimh W \cap Q = \dima E.
  \end{equation*}
  It follows from Theorem \ref{thm:BK_general} and \cite[Theorem 4.1.11]{MackayTyson2010} that
  \begin{equation*}
    \dimh W \cap Q = \cdimh W \cap Q.
  \end{equation*}
  Furthermore, by recalling \cite[Proposition 2.1]{Mackay2011}, we see that
  \begin{equation*}
    \cdima W \cap Q \le \cdima E.
  \end{equation*}
  Putting these estimates together gives
  \begin{equation*}
    \dima E \leq \dimh W \cap Q  = \cdimh W \cap Q  \le \cdima W \cap Q \le \cdima E
  \end{equation*}
  finishing the proof.
\end{proof}

To conclude the article, we discuss about the role of the condition \ref{H:projection} in Theorems \ref{thm:horizontal_carpet} and \ref{thm:conformal_dimensio}.

\begin{remark} \label{rem:projection-interval}
  In Theorem \ref{thm:conformal_dimensio}, the condition \ref{H:projection} can be replaced by the following condition:
  \begin{itemize}
    \item[(H2')] The projection of $E$ onto the horizontal coordinate is a line segment.
  \end{itemize}
  It is clearly equivalent to (H2') to assume that every vertical line that intersects $X$, the convex hull of $E$, intersects $\fii_i(X)$ for some $i\in \{ 1,\ldots,N \}$. Theorem \ref{thm:conformal_dimensio} remains true since with (H2') we can modify the proofs of Lemmas \ref{lem:endings} and \ref{lem:epspatterns} to show that the weak tangents have the form $(-\infty,w] \times \CL \cup [w,\infty) \times \CR$, where $\CL$ and $\CR$ are closed porous sets, which is enough for the conformal Hausdorff minimality of the weak tangents. Observe that the proof of Theorem \ref{thm:conformal_dimensio} actually shows that any compact set $E$, whose weak tangents are either minimal for the conformal Hausdorff dimension or minimal for the conformal Assouad dimension, is minimal for the conformal Assouad dimension. Note that this is valid in any dimension.
\end{remark}

\begin{remark} \label{rem:projections}
  In this remark, we consider the following generalization of the condition \ref{H:projection}:
  \begin{itemize}
    \item[(H2'')] Every vertical line that intersects $X$, the convex hull of $E$, either does not intersect any of the sets $\fii_i(X)$ or intersects $\fii_\ii(X)$ for at least two distinct $\ii \in \{ 1,\ldots,N \}^2$.
  \end{itemize}
  Note that according to (H2''), the projection of $E$ onto the horizontal coordinate is a finite union of line segments. These line segments are the projections of the first level images of $X$. All theorems of this section are still true if we replace \ref{H:projection} by (H2''). The key observation is that even with this weaker condition, the form of the weak tangents remains the same. In Lemmas \ref{lem:littlethings}--\ref{lem:epspatterns}, the only difference is that we have different constants. To see this, let $G$ denote the collection of the vertical lines that intersect $X$ but not $E$. Note that the images $\fii_\ii(G)$ divide the rectangle $Q_\ii$ into finitely many rectangles. These rectangles will be then used in place of the rectangles $Q_\ii$ in Lemmas \ref{lem:endings} and \ref{lem:epspatterns}. Thus there are more, but still finitely many vertical endings to deal with. Since this has an effect only on constants, the results remain the same.
\end{remark}

%%%%%%%%%%%%%%%%%%%%%%%%%%%%%%%%%%%%%%%%%%%%%%%%%%%%%%%%%%%%%%%TIKZ-PICTURE%%%%%%%%%%%%%%%%%%%%%%%%%%%%%%%%%%%%
  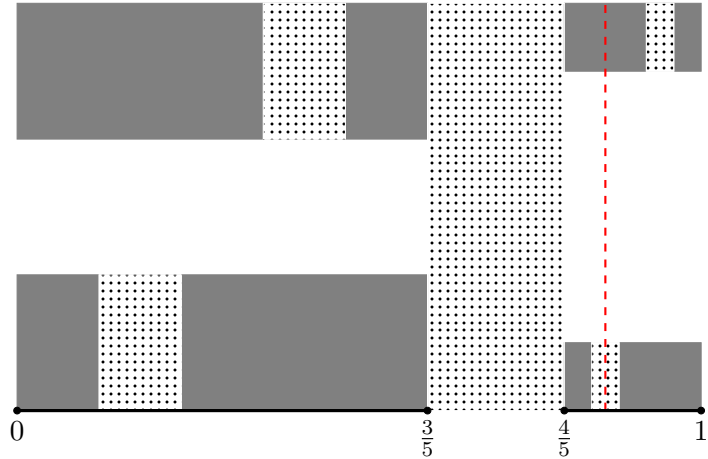
\begin{figure}[t]
    \begin{centering}
      \begin{tikzpicture}[scale=1.8]
      %draw the unite square
      %\draw (0,0) rectangle (5,3);
      %
      %draw the first level cylinders
      \foreach \j in {(0,0),(0,2)} {
          \filldraw[black!50!white] \j rectangle ++ (0.6*5,1);
      }
      \foreach \j in {(4,0),(4,2.5)} {
          \filldraw[black!50!white] \j rectangle ++ (1,0.5);
      }
      %draw a rectangles corresponding to the gaps
      %first level gap
      \draw[white,pattern=dots] (3,0) rectangle ++ (1,3);
      %second level gaps
      %gap 1
      \filldraw[white] (0.6*1,0) rectangle ++ (0.6*1,0.2*5);
      \draw[white,pattern=dots] (0.6*1,0) rectangle ++ (0.6*1,1);
      %
      % Alternative color for indicating the gap in the cylinders
      %\draw[black!50!white,pattern=north west lines] (0.6*1,0) rectangle ++ (0.6*1,1);
      %
      %gap 2
      \filldraw[white] (0.6*3,2) rectangle ++ (0.6*1,1);
      \draw[white,pattern=dots] (0.6*3,2) rectangle ++ (0.6*1,1);
      %gap 3
      \filldraw[white] (0.2*1,0) ++ (4,0) rectangle ++ (0.2*1,0.5);
      \draw[white,pattern=dots] (0.2*1,0) ++ (4,0) rectangle ++ (0.2*1,0.5);
      %gap 4
      \filldraw[white] (0.2*3,0) ++ (4,2.5) rectangle ++ (0.2*1,0.5);
      \draw[white,pattern=dots] ((0.2*3,0) ++ (4,2.5) rectangle ++ (0.2*1,0.5);
      %draw a vertical line to indicate the projection
      \draw[red,thick,dashed] (4.3,0) -- (4.3,3);
      %draw the projection at the x axis
      \draw[black,very thick] (0,0) -- (3,0);
      \draw[black,very thick] (4,0) -- (5,0);
      %draw nodes to x axix
      \draw[fill](0,0) circle (.7pt) node [below] {$0$};
      \draw[fill](3,0) circle (.7pt) node [below] {$\frac{3}{5}$};
      \draw[fill](4,0) circle (.7pt) node [below] {$\frac{4}{5}$};
      \draw[fill](5,0) circle (.7pt) node [below] {$1$};
      \end{tikzpicture}
    \end{centering}
    \caption{An example of a self-affine carpet for which the projection onto the horizontal coordinate is a union of two disjoint line segments. The gray parts represent the first level construction rectangles and the dotted areas show the gaps inside them. The dashed line shows how a vertical line, which goes through a gap of one cylinder, will hit another construction rectangle of the same level. The thick black line segments in the bottom represent the projection of the self-affine carpet in question.}
    \label{fig:projectiongap}
  \end{figure}
%%%%%%%%%%%%%%%%%%%%%%%%%%%%%%%%%%%%%%%%%%%%%%%%%%%%%%%%%%%%%%%%%%%%%%%%%%%%%%%%%%%%%%%%%%%%%%%%%%%%%%%%%%%%%%%

\begin{example} \label{ex:gap}
  Mackay \cite[Theorem 1.4]{Mackay2011} showed that for a Gazouras-Lalley carpet $E$, if the projection of $E$ onto the horizontal coordinate is not a single line segment (in which case it is a porous set and hence, (H2'') is not satisfied), then $\cdima E = 0$. For more general self-affine carpets, this is not true anymore. If a self-affine carpet $E$ satisfies the SSC, \ref{H:horizontal}, and (H2''), then, by Remark \ref{rem:projections} and Theorem \ref{thm:conformal_dimensio}, $E$ is minimal for the conformal Assouad dimension. Let us next define a self-affine carpet which satisfies these assumptions but does not satisfy \ref{H:projection}.

  Let $\Phi=\{\fii_i\}_{i=1}^4$ be such that
  \begin{align*}
  \fii_1(x_1,x_2) &= (\tfrac{3}{5} - \tfrac{3}{5}x_1 , \tfrac{1}{3}x_2), \\ 
  \fii_2(x_1,x_2) &= (\tfrac{3}{5}x_1 , \tfrac{1}{3}x_2) + (\tfrac{2}{5},0), \\
  \fii_3(x_1,x_2) &= (\tfrac{1}{5} - \tfrac{1}{5}x_1 , \tfrac{1}{6}x_2) + (\tfrac{4}{5},0), \\ 
  \fii_4(x_1,x_2) &= (\tfrac{1}{5}x_1 , \tfrac{1}{6}x_2) + (\tfrac{4}{5},\tfrac{2}{5}).
  \end{align*}
  All the maps map the rectangle $X=[0,1]\times[0,\frac{3}{5}]$ into itself; see Figure \ref{fig:projectiongap}. Observe that the maps $\fii_1$ and $\fii_3$ contain a reflection in the first coordinate. We see that each vertical line that intersects $\fii_i(X)$ for some $i$ also intersects $\fii_{jk}(X)$ for some $j$ and $k$. Thus the condition (H2'') is satisfied. By Remark \ref{rem:projections}, the vertical projection of $E$ equals to the union of the projections of $\fii_i(X)$, which in this case, is a union of two disjoint line segments.
\end{example}

\bibliographystyle{abbrv}
%\bibliography{../../StyAndBib/Bibliography}
\bibliography{Bibliography}

\begin{thebibliography}{10}

\bibitem{BandtKaenmaki2013}
C.~Bandt and A.~K{\"a}enm{\"a}ki.
\newblock {Local structure of self-affine sets}.
\newblock {\em Ergodic Theory Dynam. Systems}, 33(5):1326--1337, 2013.

\bibitem{BishopPeres2016}
C.~J. Bishop and Y.~Peres.
\newblock {Fractals in Probability and Analysis}.
\newblock Unpublished, 2016.

\bibitem{BonkMerenkov2013}
M.~Bonk and S.~Merenkov.
\newblock {Quasisymmetric rigidity of square {S}ierpi{\'n}ski carpets}.
\newblock {\em Ann. of Math. (2)}, 177(2):591--643, 2013.

\bibitem{DavidSemmes1997}
G.~David and S.~Semmes.
\newblock {\em {Fractured fractals and broken dreams}}, volume~7 of {\em
  {Oxford Lecture Series in Mathematics and its Applications}}.
\newblock The Clarendon Press, Oxford University Press, New York, 1997.
\newblock Self-similar geometry through metric and measure.

\bibitem{Heinonen2001}
J.~Heinonen.
\newblock {\em {Lectures on analysis on metric spaces}}.
\newblock {Universitext}. Springer-Verlag, New York, 2001.

\bibitem{Hutchinson1981}
J.~E. Hutchinson.
\newblock {Fractals and self-similarity}.
\newblock {\em Indiana Univ. Math. J.}, 30(5):713--747, 1981.

\bibitem{KaenmakiKoivusaloRossi2015}
A.~K{\"a}enm{\"a}ki, H.~Koivusalo, and E.~Rossi.
\newblock {Self-affine sets with fibred tangents}.
\newblock Ergodic Theory and Dynam. Systems, to appear, available at
  arXiv:1505.00958.

\bibitem{KaenmakiRossi2015}
A.~K{\"a}enm{\"a}ki and E.~Rossi.
\newblock {Weak separation condition, Assouad dimension, and Furstenberg
  homogenity}.
\newblock {\em Ann. Acad. Sci. Fenn. Math.}, 41:465--490, 2016.

\bibitem{LeDonneRajala2015}
E.~Le~Donne and T.~Rajala.
\newblock Assouad dimension, {N}agata dimension, and uniformly close metric
  tangents.
\newblock {\em Indiana Univ. Math. J.}, 64(1):21--54, 2015.

\bibitem{LeDonneXie2015}
E.~{Le Donne} and X.~Xie.
\newblock {Rigidity of fiber-preserving quasisymmetric maps}.
\newblock Rev. Mat. Iberoam., to appear, available at arXiv:1501.02391.

\bibitem{Mackay2011}
J.~M. Mackay.
\newblock {Assouad dimension of self-affine carpets}.
\newblock {\em Conform. Geom. Dyn.}, 15:177--187, 2011.

\bibitem{MackayTyson2010}
J.~M. Mackay and J.~T. Tyson.
\newblock {\em {Conformal dimension}}, volume~54 of {\em {University Lecture
  Series}}.
\newblock American Mathematical Society, Providence, RI, 2010.
\newblock Theory and application.

\bibitem{TukiaVaisala1980}
P.~Tukia and J.~V{\"a}is{\"a}l{\"a}.
\newblock Quasisymmetric embeddings of metric spaces.
\newblock {\em Ann. Acad. Sci. Fenn. Ser. A I Math.}, 5(1):97--114, 1980.

\bibitem{WangWenZhu2010}
X.~Wang, S.~Wen, and C.~Zhu.
\newblock {Quasisymmetric equivalence of self-similar sets}.
\newblock {\em J. Math. Anal. Appl.}, 365(1):254--258, 2010.

\end{thebibliography}

\end{document}